\documentclass[article]{IEEEtran}

\usepackage[english]{babel}
\usepackage[latin1]{inputenc}
\usepackage{enumerate}
\usepackage{color}
\usepackage[T1]{fontenc}
\usepackage{epstopdf}
\usepackage{subfigure}
\usepackage{dsfont}
\usepackage[final]{graphicx}
\DeclareGraphicsExtensions{.eps}
\usepackage[T1]{fontenc}
\usepackage{amsmath}
\usepackage{mathtools}
\usepackage{amsthm}
\usepackage{amstext}
\usepackage{amssymb}
\usepackage{mathrsfs}
\usepackage{cite}
\usepackage{mathtools}
\usepackage{tikz}
\usetikzlibrary{arrows,positioning}

\def\R{\mathbb{R}}
\def\rP{\mathbb{P}}

\def\intr{\mathop{\rm int}}

\def\argmin{\mathop{\rm arg\, min}}

\def\S{{\mathcal S}}

\def\by{{\bf y}}
\def\bu{{\bf u}}

\def\sq{{\mathsf q}}
\def\sX{{\mathsf X}}
\def\sY{{\mathsf Y}}

\def\sZ{{\mathsf Z}}

\def\sE{{\mathsf E}}

\def\sU{{\mathsf U}}
\def\sK{{\mathsf K}}
\def\sg{{\mathsf g}}

\def\tgamma{\tilde{\gamma}}

\theoremstyle{remark}

\newtheorem{definition}{Definition}
\newtheorem{theorem}{Theorem}

\newtheorem{proposition}{Proposition}
\newtheorem{lemma}{Lemma}
\theoremstyle{remark}
\newtheorem{remark}{Remark}
\newtheorem{assumption}{Assumption}

\IEEEoverridecommandlockouts

\allowdisplaybreaks

\begin{document}
\sloppy
\title{Finite Model Approximations and Asymptotic Optimality of Quantized Policies in Decentralized Stochastic Control
\thanks{This research was supported in part by
the Natural Sciences and Engineering Research Council (NSERC) of Canada.}
}
\author{Naci Saldi, Serdar Y\"uksel, Tam\'{a}s Linder
\thanks{The authors are with the Department of Mathematics and Statistics,
     Queen's University, Kingston, ON, Canada,
     Email: \{nsaldi,yuksel,linder\}@mast.queensu.ca}
     }
\maketitle

\begin{abstract}
In this paper, we consider finite model approximations of a large class of static and dynamic team problems where these models are constructed through uniform quantization of the observation and action spaces of the agents. The strategies obtained from these finite models are shown to approximate the optimal cost with arbitrary precision under mild technical assumptions. In particular, quantized team policies are asymptotically optimal. This result is then applied to Witsenhausen's celebrated counterexample and the Gaussian relay channel problem. For the Witsenhausen's counterexample, our approximation approach provides, to our knowledge, the first rigorously established result that one can construct an $\varepsilon$-optimal strategy for any $\varepsilon > 0$ through a solution of a simpler problem.
\end{abstract}

\section{Introduction}\label{sec1}

Team decision theory has its roots in control theory and economics. Marschak \cite{mar55} was perhaps the first to introduce the basic elements of teams and to  provide the first steps toward the development of a {\it team theory}. Radner \cite{rad62} provided foundational results for static teams, establishing connections between person-by-person optimality, stationarity, and team-optimality. The work of Witsenhausen \cite{wit71}, \cite{wit75}, \cite{wit88}, \cite{WitsenStandard}, \cite{WitsenhausenSIAM71} on dynamic teams and characterization of information structures has been crucial in the progress of our understanding of dynamic teams. Further discussion on the design of information structures in the context of team theory and economics applications are given in \cite{Arrow1985} and \cite{vanZandt}, among a rich collection of other contributions not listed here.

Establishing the existence and structure of optimal policies is a challenging problem. Existence of optimal policies for static and a class of sequential dynamic teams has been studied recently in \cite{GuYuBaLa15}. More specific setups and non-existence results have been studied in \cite{WuVer11}, \cite{wit68}. For a class of teams which are {\it convex}, one can reduce the search space to a smaller parametric class of policies (see \cite{rad62,KraMar82,WaSc00}, and for a comprehensive review, see \cite{YukselBasarBook}).

In this paper, our aim is to study the approximation of static and dynamic team problems using finite models which are obtained through the uniform discretization, on a finite grid, of the observation and action spaces of agents. In particular, we are interested in the asymptotic optimality of quantized policies.

In the literature relatively few results are available on approximating static or dynamic team problems. We can only refer the reader to \cite{LiMaSh09},\cite{LeLaHo01,BaPaZo01,MeAkRo14,GnSa09,GnSa12,GnSaGa12} and a few references therein. With the exception of \cite{GnSa09,GnSa12,GnSaGa12}, these works in general study a specific setup (the Witsenhausen counterexample) and are mostly experimental; that is, they do not rigorously prove the convergence of approximate solutions.

In \cite{GnSa09,GnSaGa12}, a class of static team problems are considered and the existence of smooth optimal strategies is studied. Under fairly strong assumptions, the existence of an optimal strategy with Lipschitz continuous partial derivatives up to some order is proved. By using this result, an error bound on the accuracy of near optimal solutions is established, where near optimal strategies are expressed as linear combinations of basis functions with adjustable parameters. In \cite{GnSa12}, the same authors investigated the approximation problem for Witsenhausen's counterexample, which does not satisfy the conditions in \cite{GnSa09,GnSaGa12}; the authors derived an analogous error bound on the accuracy of the near optimal solutions. For the result in \cite{GnSa12} both the error bound and the near optimal solutions depend on the knowledge of the optimal strategy for Witsenhausen's counterexample. Moreover, the method devised in \cite{GnSa12} implicitly corresponds to the discretization of only the action spaces of the agents. Therefore, it involves only the approximation with regard to the action space, and does not correspond to a tractable approximation for the set of policies/strategies.

Particular attention has been paid in the literature to Witsenhausen's counterexample. This problem has puzzled the control community for more than 40 years with its philosophical impact demonstrating the challenges that arise due to a non-classical information structure, and its formidable difficulty in obtaining an optimal or suboptimal solution. In fact, optimal policies and their value are still unknown, even though the existence of an optimal policy has been established using various methods \cite{GuYuBaLa15,wit68,WuVer11}. Some relevant results on obtaining approximate solutions can be found in \cite{McHa15,LiMaSh09,LeLaHo01,BaPaZo01,MeAkRo14,GnSa12,banbas87a}. Certain lower bounds, that are not tight, building on information theoretic approaches are available in \cite{GroverParkSahai}, see also \cite{MitSah99}. In this paper, we show that finite models obtained through the uniform quantization of the observation and action spaces lead to a sequence of policies whose cost values converge to the optimal cost. Thus, with high enough computation power, one could guarantee that for any $\varepsilon > 0$, an $\varepsilon$-optimal policy can be constructed.

We note that the operation of quantization has typically been the method to show that a non-linear policy can perform better than an optimal linear policy, both for Witsenhausen's counterexample \cite{banbas87a,LeLaHo01} and another interesting model known as the Gaussian relay channel problem \cite{LipsaMartins,zaidi2013optimal}. Our findings show that for a large class of problems, quantized policies not only may perform better than linear policies, but that they are actually almost optimal.

We finally note that although finding optimal solutions for finite models for Witsenhausen's counterexample as the ones constructed in this paper was shown to be NP-complete in \cite{PaTs86}, the task is still computationally less demanding than the method used in \cite{GnSa12}. Loosely speaking, to obtain a near optimal solution using the method in \cite{GnSa12}, one has to compute the optimal partitions of the observation spaces and the optimal representation points in the action spaces. In contrast, the partitions of the observation spaces and the available representation points in the action spaces used by our method are known \emph{a priori}. We also note that if one can establish smoothness properties of optimal policies such as differentiability or Lipschitz continuity (e.g., as in \cite{GnSaGa12}), the methods developed in our paper can be used to provide rates of convergence for the sequence of finite solutions as the finite models are successively refined.


{\bf Contributions of the paper.} (i) We establish that finite models asymptotically represent the true models in the sense that the solutions obtained by solving such finite models lead to cost values that converge to the optimal cost of the original model. Thus, our approach can be viewed to be constructive; even though the computational complexity is typically at least exponential in the cardinality of the finite model. (ii) The approximation approach here provides, to our knowledge, the first rigorously established result showing that one can construct an $\varepsilon$-optimal strategy for any $\varepsilon > 0$ through an explicit solution of a simpler problem for a large class of static and dynamic team problems, in particular for the Witsenhausen's celebrated counterexample.

The rest of the paper is organized as follows. In Section~\ref{sub1sec1} we review the definition of Witsenhausen's \emph{intrinsic model} for sequential team problems. In Section~\ref{boundedcase} we consider finite \emph{observation} approximations of static team problems with compact observation spaces and bounded cost, and prove the asymptotic optimality of strategies obtained from finite models. In Section~\ref{unboundedcase} an analogous approximation result is obtained for static team problems with non-compact observation spaces and unbounded cost functions. In Section~\ref{sec3} we consider finite observation approximations of dynamic team problems via the static reduction method. In Sections~\ref{sec4} and \ref{gaussrelay} we apply the results derived in Section~\ref{sec3} to study finite observation space approximations of Witsenhausen's celebrated counterexample and the Gaussian relay channel. Discretization of the action spaces is considered in Section~\ref{discact}. Section~\ref{sec5} concludes the paper.

\section{Sequential teams and characterization of information structures}\label{sub1sec1}

In this section, we introduce the model as laid out by Witsenhausen, called {\it the Intrinsic Model} \cite{wit75}; see \cite{YukselBasarBook} for a more comprehensive overview and further characterizations and classifications of information structures. In this model, any action applied at any given time is regarded as applied by an individual agent, who acts only once. One advantage of this model, in addition to its generality, is that the definitions regarding information structures can be compactly described.

Suppose that in the decentralized system, there is a pre-defined order in which the agents act. Such systems are called {\it sequential systems} (for non-sequential teams, we refer the reader to \cite{AnderslandTeneketzisI}, \cite{AnderslandTeneketzisII} and \cite{Teneketzis2}, in addition to \cite{WitsenhausenSIAM71}). In the following, all spaces are assumed to be Borel spaces (i.e., Borel subsets of complete and separable metric spaces) endowed with Borel $\sigma$-algebras. In the context of a sequential system, the {\it Intrinsic Model} has the following components:

\begin{itemize}
\item A collection of {\it measurable spaces} ${\cal I}:= \bigl\{  (\Omega, {\cal F}),
(\sU^i,{\cal U}^i), (\sY^i,{\cal Y}^i), i \in {\cal N}\bigr\}$, specifying the system's distinguishable events, and the action and measurement spaces. Here $|{\cal N}| \coloneqq N$ is the number of actions taken, and each of these actions is supposed to be taken by an individual agent (hence, an agent with perfect recall can also be regarded as a separate decision maker every time it acts). The pair $(\Omega, {\cal F})$ is a measurable space on which an underlying probability can be defined. The pair $(\sU^i, {\cal U}^i)$ denotes the measurable space from which the action $u^i$ of Agent~$i$ is selected. The pair $(\sY^i,{\cal Y}^i)$ denotes the measurement (or observation) space of Agent~$i$.

\item A {\it measurement constraint} which establishes the connection between the observation variables and the system's distinguishable events. The $\sY^i$-valued observation variables are given by $y^i \sim \eta^i(\,\cdot\,|\omega,{\bf u}^{i-1})$, where ${\bf u}^{i-1}=(u^1,\ldots,u^{i-1})$, $\eta^i$ is a stochastic kernel on $\sY^i$ given $\Omega\times \prod_{j=1}^{i-1}\sU^j$, and $u^k$ denotes the action of Agent~$k$.

\item A {\it design constraint}, which restricts the set of admissible $N$-tuple control strategies $\underline{\gamma}= \{\gamma^1, \gamma^2, \dots, \gamma^N\}$, also called {\it policies}, to the set of all measurable functions, so that $u^i = \gamma^i(y^i)$, where $\gamma^i$ is a measurable function. Let $\Gamma^i$ denote the set of all admissible policies for Agent~$i$ and let ${\bf \Gamma} = \prod_{k} \Gamma^k$.

\item A {\it probability measure} $\rP$ defined on $(\Omega, {\cal F})$ which describes the measures on the random events in the model.
\end{itemize}

We note that the intrinsic model of Witsenhausen gives a set-theoretic characterization of information fields; however, for Borel spaces, the model above is equivalent to the intrinsic model for sequential team problems.

Under this intrinsic model, a sequential team problem is {\it dynamic} if the information available to at least one agent is affected by the action of at least one other agent. A decentralized problem is {\it static}, if the information available at every decision maker is only affected by state of the nature; that is, no other decision maker can affect the information at any given decision maker.

Information structures (ISs) can also be classified as classical, quasi-classical, and nonclassical. An IS is {\it classical} if $y^i$ contains all of the information available to Agent~$k$ for $k < i$. An IS is {\it quasi-classical} or {\it partially nested}, if whenever $u^k$ (for some $k < i$) affects $y^i$, then Agent~$i$ has access to $y^k$. An IS which is not partially nested is {\it nonclassical}.

For any $\underline{\gamma} = (\gamma^1, \cdots, \gamma^N)$, we let the cost of the team problem be defined by
\[J(\underline{\gamma}) \coloneqq E[c(\omega,{\bf y},{\bf u})],\]
for some measurable cost function $c: \Omega \times \prod_i \sY^i \times \prod_i \sU^i \to \mathbb{R}$, where ${\bf u} \coloneqq (u^1,\ldots,u^N) = \underline{\gamma}({\bf y})$ and ${\bf y} \coloneqq (y^1,\ldots,y^N)$.

\begin{definition}\label{Def:TB1}
For a given stochastic team problem, a policy (strategy)
${\underline \gamma}^*:=({\gamma^1}^*,\ldots, {\gamma^N}^*)\in {\bf \Gamma}$ is
an {\it optimal team decision rule} if
\begin{equation}
J({\underline \gamma}^*)=\inf_{{{\underline \gamma}}\in {{\bf \Gamma}}}
J({{\underline \gamma}})=:J^*. \nonumber
\end{equation}
The cost level $J^*$ achieved by this strategy is the {\it optimal team cost}.
\end{definition}

\begin{definition}\label{Def:TB2}
For a given stochastic team problem, a policy ${\underline \gamma}^*:=({\gamma^1}^*,\ldots, {\gamma^N}^*)$ constitutes a {\it Nash
equilibrium} (synonymously,  a {\it person-by-person optimal} solution) if, for all $\beta
\in \Gamma^i$ and all $i\in N$, the following inequalities hold:
\begin{equation}
{J}^*:=J({\underline \gamma}^*) \leq J({\underline \gamma}^{-i*},
\beta), \nonumber
\end{equation}
where we have adopted the notation
\begin{equation}
({\underline \gamma}^{-i*},\beta):= ({\gamma^1}^*,\ldots, {\gamma^{i-1}}^*,
\beta, {\gamma^{i+1}}^*,\ldots, {\gamma^N}^*). \nonumber
\end{equation}
\end{definition}

Unless otherwise specified, the term `measurable' will refer to Borel measurability in the rest of the paper. In what follows, the terms \emph{policy}, \emph{measurement}, and \emph{agent} are used synonymously with \emph{strategy}, \emph{observation}, and \emph{decision maker}, respectively.

\subsection{Auxiliary Results}

In this section, for the ease of reference we state some well-known results in measure theory and functional analysis that will be frequently used in the paper.

The first theorem is Lusin's theorem which roughly states that any measurable function is \emph{almost} continuous.

\begin{theorem}{(Lusin's Theorem \cite[Theorem 7.5.2]{Dud02})}\label{lusin}
Let $\sE_1$ and $\sE_2$ be two Borel spaces and let $\mu$ be a probability measure on $\sE_1$. Let $f$ be a measurable function from $\sE_1$ into $\sE_2$. Then, for any $\varepsilon>0$ there is a closed set $F \subset \sE_1$ such that $\mu(\sE_1 \setminus F) < \varepsilon$ and the restriction of $f$ to $F$ is continuous.
\end{theorem}

The second theorem is the Dugundji extension theorem which is a generalization of the Tietze extension theorem \cite{Dud02}.

\begin{theorem}{(Dugundji Extension Theorem \cite[Theorem 7.4]{GrDu03})}\label{dugundji}
Let $\sE_1$ be a Borel space and let $F$ be a closed subset of $\sE_1$. Let $\sE_2$ be a convex subset of some locally convex vector space.
Then any continuous $f:F \rightarrow \sE_2$ has a continuous extension on $\sE_1$.
\end{theorem}

The next theorem originally states that the closed convex hull of a compact subset in a locally convex vector space is compact if the vector space is completely metrizable. Since the closure of a convex set is convex and a closed subset of a compact set is compact, we can state the theorem in the following form.

\begin{theorem}{\cite[Theorem 5.35]{AlBo06}}\label{clconv}
In a completely metrizable locally convex vector space $\sE$, the closed convex hull of a compact set is convex and compact. The same statement also holds when $\sE$ is replaced with any of its closed and convex subsets.
\end{theorem}

\section{Approximation of Static Team Problems}\label{sec2}

In this section, we consider the finite observation approximation of static team problems. In what follows, the static team problem is formulated in a state-space form which can be reduced to the intrinsic model introduced in Section~\ref{sub1sec1}.

Let $\bigl(\sX,{\cal X},\rP)$ be a probability space representing the state space, where $\sX$ is a Borel space and ${\cal X}$ is its Borel $\sigma$-algebra. We consider an $N$-agent static team problem in which Agent $i$, $i=1,\ldots,N$, observes a random variable $y^i$ and takes an action $u^i$, where $y^i$ takes values in a Borel space $\sY^i$ and $u^i$ takes values in a Borel space $\sU^i$. Given any state realization $x$, the random variable $y^i$ has a distribution $W^i(\,\cdot\,|x)$; that is, $W^i(\,\cdot\,|x)$ is a stochastic kernel on $\sY^i$ given $\sX$ \cite{HeLa96}.

The team cost function $c$ is a non-negative function of the state, observations, and actions; that is, $c: \sX \times \sY \times \sU \rightarrow [0,\infty)$, where $\sY \coloneqq \prod_{i=1}^N \sY^i$ and $\sU \coloneqq \prod_{i=1}^N \sU^i$. For Agent $i$, the set of strategies $\Gamma^i$ is given by
\begin{align}
\Gamma^i \coloneqq \bigl\{ \gamma^i: \sY^i \rightarrow \sU^i, \text{$\gamma^i$ is measurable} \bigr\}. \nonumber
\end{align}
Recall that ${\bf \Gamma} = \prod_{i=1}^N \Gamma^i$. Then, the cost of the team $J: {\bf \Gamma} \rightarrow [0,\infty)$ is given by
\begin{align}
J(\underline{\gamma}) = \int_{\sX \times \sY} c(x,{\bf y},{\bf u}) \rP(dx,d{\bf y}), \nonumber
\end{align}
where ${\bf u} = \underline{\gamma}({\bf y})$. Here, with an abuse of notation, $\rP(dx,d{\bf y}) \coloneqq ~\rP(dx) \prod_{i=1}^N W^i(dy^i|x)$ denotes the joint distribution of the state and observations. Therefore, we have
\begin{align}
J^* = \inf_{\underline{\gamma} \in {\bf \Gamma} } J(\underline{\gamma}). \nonumber
\end{align}

With these definitions, we first consider the case where the observation spaces are compact and the cost function is bounded. In the second part, teams with non-compact observation spaces and unbounded cost function will be studied.

\subsection{Approximation of Static Teams with Compact Observation Spaces and Bounded Cost}\label{boundedcase}

In this section, we consider the finite observation approximation of static team problems with compact observation spaces and bounded cost function. We impose the following assumptions on the components of the model.

\begin{assumption}\label{as1}
\begin{itemize}
\item [ ]
\item [(a)] The cost function $c$ is bounded. In addition, it is continuous in ${\bf u}$ and ${\bf y}$ for any fixed $x$.
\item [(b)] For each $i$, $\sU^i$ is a convex subset of a locally convex vector space.
\item [(c)] For each $i$, $\sY^i$ is compact.
\end{itemize}
\end{assumption}

We first prove that the minimum cost achievable by continuous strategies is equal to the optimal cost $J^*$. To this end, for each $i$, we define
$\Gamma^i_c \coloneqq \bigl\{ \gamma^i \in \Gamma^i : \text{ $\gamma^i$ is continuous}\bigr\}$ and ${\bf \Gamma}_c \coloneqq \prod_{i=1}^N \Gamma^i_c$.

\begin{proposition}\label{prop1}
We have
\begin{align}
\inf_{\underline{\gamma} \in {\bf \Gamma}_c} J(\underline{\gamma}) = J^*. \nonumber
\end{align}
\end{proposition}

\begin{proof}
Let $\underline{\gamma} \in {\bf \Gamma}$. We prove that there exists a sequence $\bigl\{\underline{\gamma}_{k}\bigr\}_{k\geq1} \in {\bf \Gamma}_c$ such that $J(\underline{\gamma}_{k}) \rightarrow J(\underline{\gamma})$ as $k\rightarrow\infty$, which implies the proposition. Let $\mu^i$ denote the distribution of $y^i$.

For each $k\geq1$, by Lusin's theorem, there is a closed set $F_{k,i} \subset \sY^i$ such that $\mu^i\bigl( \sY^i \setminus F_{k,i} \bigr) < 1/k$ and the restriction of $\gamma^i$ to $F_{k,i}$ is continuous. Let us denote $\pi^i_k = \gamma^{i}\bigl|_{F_{k,i}}$, and so $\pi^i_{k}: F_{k,i} \rightarrow \sU^i$ is continuous. By the Dugundji extension theorem, there exists a continuous extension $\gamma^i_k: \sY^i \rightarrow \sU^i$ of $\pi^i_{k}$. Therefore, $\underline{\gamma}_{k} = (\gamma^1_k,\ldots,\gamma^N_k) \in {\bf \Gamma}_c$ and we have for $F_k \coloneqq \prod_{i=1}^N F_{k,i}$ the following
\begin{align}
\bigl| J(\underline{\gamma}) - &J(\underline{\gamma}_{k}) \bigr|
= \biggl| \int_{\sX \times \sY} \bigl[ c(x,{\bf y},\underline{\gamma}) - c(x,{\bf y},\underline{\gamma}_{k}) \bigr] \rP(dx,d{\bf y}) \biggr| \nonumber \\
&\leq \int_{\sX \times (\sY \setminus F_k)} \bigl| c(x,{\bf y},\underline{\gamma}) - c(x,{\bf y},\underline{\gamma}_{k}) \bigr| \text{  }\rP(dx,d{\bf y}) \nonumber \\
&\leq  2 \|c\| \text{ } \rP\bigl(\sX \times (\sY \setminus F_k)\bigr), \nonumber
\end{align}
where $\|c\|$ is the maximum absolute value that $c$ takes. Since $\rP\bigl(\sX \times (\sY \setminus F_k)\bigr) \leq \sum_{i=1}^N \mu^i\bigl( \sY^i \setminus F_{k,i} \bigr) \leq N/k$, we have $\lim_{k\rightarrow\infty} J(\underline{\gamma}_{k}) = J(\underline{\gamma})$. This completes the proof.
\end{proof}

Let $d_i$ denote the metric on $\sY^i$. Since $\sY^i$ is compact, one can find a finite set $\sY_n^{i} \coloneqq \bigl\{ y_{i,1}, \ldots, y_{i,i_n} \bigr\} \subset \sY^i$ such that $\sY_n^{i}$ is an $1/n$-net in $\sY^i$; that is, for any $y \in \sY^i$ we have
\begin{align}
\min_{z \in \sY_n^{i}} d_i(y,z) < \frac{1}{n}. \nonumber
\end{align}
Define function $q_{n}^i$ mapping $\sY^i$ to $\sY_n^{i}$ by
\begin{align}
q_{n}^i(y) \coloneqq \argmin_{z \in \sY_n^{i}} d_{i}(y,z), \nonumber
\end{align}
where ties are broken so that $q_n^i$ is measurable. In the literature, $q_{n}^i$ is called the nearest neighborhood quantizer \cite{GrNe98}. If $\sY^i=[-M,M]$ for some $M \in \R_+$, one can choose the finite set $\sY_n^{i}$ such that $q_{n}^i$ becomes a uniform quantizer.
For any $\gamma^i \in \Gamma^i$, we let $\gamma^{n,i}$ denote the strategy $\gamma^i \circ q_{n}^i$.
Define
\begin{align}
\Gamma_n^{i} \coloneqq  \Gamma^i \circ q_n^i.
\nonumber
\end{align}
We let ${\bf \Gamma}_n \coloneqq \prod_{i=1}^N \Gamma_n^{i}$. The following theorem states that an optimal (or almost optimal) policy can be approximated with arbitrarily small approximation error for the induced costs by policies in ${\bf \Gamma}_n$, for $n$ sufficiently large.

\begin{theorem}\label{thm1}
We have
\begin{align}
\lim_{n\rightarrow\infty} \inf_{\underline{\gamma} \in {\bf \Gamma}_n} J(\underline{\gamma}) = J^*. \nonumber
\end{align}
\end{theorem}

\begin{proof}
For any $\varepsilon$, let $\underline{\gamma}_{\varepsilon} = (\gamma_{\varepsilon}^1,\ldots,\gamma_{\varepsilon}^N) \in {\bf \Gamma}_c$ denote an $\varepsilon$-optimal continuous strategy. The existence of such a strategy follows from Proposition~\ref{prop1}. Then, we have
\begin{align}
\inf_{\underline{\gamma} \in {\bf \Gamma}_n} J(\underline{\gamma}) - J^* &=  \inf_{\underline{\gamma} \in {\bf \Gamma}_n} J(\underline{\gamma}) - \inf_{\underline{\gamma} \in {\bf \Gamma}_c} J(\underline{\gamma}) \text{ } \text{(by Proposition~\ref{prop1})}\nonumber \\
&\leq  J(\underline{\gamma}_{\varepsilon,n}) - \inf_{\underline{\gamma} \in {\bf \Gamma}_c} J(\underline{\gamma}) \nonumber \\
&\leq \varepsilon + \bigl( J(\underline{\gamma}_{\varepsilon,n}) - J(\underline{\gamma}_{\varepsilon}) \bigr), \nonumber
\end{align}
where $\underline{\gamma}_{\varepsilon,n} = (\gamma_{\varepsilon}^{n,1},\ldots,\gamma_{\varepsilon}^{n,N})$. Note that $c(x,{\bf y},\underline{\gamma}_{\varepsilon,n}({\bf y})) \rightarrow c(x,{\bf y},\underline{\gamma}_{\varepsilon}({\bf y}))$ as $n \rightarrow \infty$, for all $(x,{\bf y}) \in \sX\times\sY$ since $c$ is continuous in ${\bf u}$ and $\underline{\gamma}_{\varepsilon} \in {\bf \Gamma}_c$. Hence, by the dominated convergence theorem
the second term in the last expression converges to zero as $n\rightarrow\infty$. Since $\varepsilon$ is arbitrary, this completes the proof.
\end{proof}

For each $n$, define stochastic kernels $W_n^{i}(\,\cdot\,|x)$ on $\sY_n^{i}$ given $\sX$ as follows:
\begin{align}
W_n^{i}(\,\cdot\,|x) &\coloneqq \sum_{j=1}^{i_n} W(\S_{i,j}^n|x) \delta_{y_{i,j}}(\,\cdot\,), \nonumber
\end{align}
where $\S_{i,j}^n \coloneqq \bigl\{y \in \sY^i: q_n^i(y) = y_{i,j} \bigr\}$. Let $\Pi_n^{i} \coloneqq \bigl\{\pi^i: \sY_n^{i} \rightarrow \sU^i, \text{$\pi^i$ measurable}\bigr\}$ and  ${\bf \Pi}_n \coloneqq \prod_{i=1}^N \Pi_n^{i}$. Define $J_n: {\bf \Pi}_n \rightarrow [0,\infty)$ as
\begin{align}
J_n(\underline{\pi}) \coloneqq \int_{\sX \times \sY_n} c(x,{\bf y},{\bf u}) \rP_n(dx,d{\bf y}), \nonumber
\end{align}
where $\underline{\pi} = (\pi^1,\ldots,\pi^N)$, ${\bf u} = \underline{\pi}({\bf y})$, $\sY_n = \prod_{i=1}^N \sY_n^{i}$, and $\rP_n(dx,d{\bf y}) = \rP(dx) \prod_{i=1}^N W_n^{i}(dy^{i}|x)$.

\begin{lemma}\label{convergence-comp}
Let $\{\underline{\pi}_{n}\}$ be a sequence of strategies such that $\underline{\pi}_{n} \in {\bf \Pi}_{n}$. For each $n$, define $\underline{\gamma}_{n} \coloneqq \underline{\pi}_{n} \circ q_n$, where $q_n \coloneqq (q_n^1,\ldots,q_n^N)$. Then, we have
\begin{align}
\lim_{n\rightarrow\infty} |J_n(\underline{\pi}_n) - J(\underline{\gamma}_n)| =  0. \nonumber
\end{align}
\end{lemma}
\begin{proof}
We have
\begin{align}
&|J_{n}(\underline{\pi}_n) - J(\underline{\gamma}_n)| \nonumber \\
&\phantom{xxxxx}= \bigg| \int_{\sX \times \sY} c(x,q_{n}({\bf y}),\underline{\gamma}_n) \text{ } d\rP - \int_{\sX \times \sY} c(x,{\boldsymbol y},\underline{\gamma}_n) \text{ } d\rP \biggr| \nonumber
\end{align}
which converges to zero as $n\rightarrow\infty$ by dominated convergence theorem and the fact that $c$ is bounded and continuous in ${\bf y}$.
\end{proof}

The following theorem is the main result of this section which is a consequence of Theorem~\ref{thm1}. It states that to compute a near optimal strategy for the original team problem, it is sufficient to compute an optimal (or almost optimal if optimal does not exist) strategy for the team problem described above.

\begin{theorem}\label{thm2}
For any $\varepsilon>0$, there exists a sufficiently large $n$ such that the optimal (or almost optimal) policy $\underline{\pi}^{*} \in {\bf \Pi}_n$ for the cost $J_n$ is $\varepsilon$-optimal for the original team problem when $\underline{\pi}^{*} = (\pi^{1*},\ldots,\pi^{N*})$ is extended to $\sY$ via $\gamma^i = \pi^{i*} \circ q_{n}^i$.
\end{theorem}

\begin{proof}
Fix any $\varepsilon>0$. By Theorem~\ref{thm1}, there exists a sequence of strategies $\{\underline{\gamma}_n\}$ such that
$\underline{\gamma}_n \in {\bf \Gamma}_n$ ($n\geq1$) and $ \lim_{n\rightarrow\infty} J(\underline{\gamma}_n) = J^*$. Define $\underline{\pi}_n$ as the restriction of $\underline{\gamma}_n$ to the set $\sY_n$. Then, we have
\begin{align}
J^* &= \lim_{n\rightarrow\infty} J(\underline{\gamma}_n) \nonumber \\
&= \lim_{n\rightarrow\infty} J_n(\underline{\pi}_n) \text{ } \text{(by Lemma~\ref{convergence-comp})} \nonumber \\
&\geq \limsup_{n\rightarrow\infty} \inf_{\underline{\pi} \in {\bf \Pi}_{n}} J_n(\underline{\pi}). \nonumber
\end{align}
For the reverse inequality, for each $n\geq1$, let $\underline{\pi}_n \in {\bf \Pi}_n$ be such that $J_n(\underline{\pi}_n) < \inf_{\underline{\pi} \in {\bf \Pi}_n} J_n(\underline{\pi}) + 1/n$. Define $\underline{\gamma}_n \coloneqq \underline{\pi}_n \circ q_n$.
Then, we have
\begin{align}
\liminf_{n\rightarrow\infty} \inf_{\underline{\pi} \in {\bf \Pi}_n} J_n(\underline{\pi}) &= \liminf_{n\rightarrow\infty} J_n(\underline{\pi}_n) \nonumber \\
&= \liminf_{n\rightarrow\infty} J(\underline{\gamma}_n) \text{ } \text{(by Lemma~\ref{convergence-comp})} \nonumber \\
&\geq J^*. \nonumber
\end{align}
This completes the proof.
\end{proof}

\subsection{Approximation of Static Teams with Noncompact Observation Spaces and Unbounded Cost}\label{unboundedcase}

In this section, we consider the finite observation approximation of static team problems with noncompact observation spaces and unbounded cost function.
We impose the following assumptions on the components of the model.

\begin{assumption}\label{newas1}
\begin{itemize}
\item [ ]
\item [(a)] The cost function $c$ is continuous in ${\bf u}$ and ${\bf y}$ for any fixed $x$. In addition, it is bounded on any compact subset of $\sX \times \sY \times \sU$. 
\item [(b)] For each $i$, $\sU^i$ is a closed and convex subset of a completely metrizable locally convex vector space.
\item [(c)] For each $i$, $\sY^i$ is locally compact.
\item [(d)] For any subset $G$ of $\sU$, we let $w_G(x,\by) \coloneqq \sup_{\bu \in G} c(x,\by,\bu)$. We assume that $w_G$ is integrable with respect to $\rP(dx,d\by)$, for any compact subset $G$ of $\sU$ of the form $G = \prod_{i=1}^N G^i$.
\item [(e)] For any $\underline{\gamma} \in {\bf \Gamma}$ with $J(\underline{\gamma})<\infty$ and each $i \in {\cal N}$, there exists $u^{i,*} \in \sU^i$ such that we have $J(\underline{\gamma}^{-i},\gamma^i_{u^{i,*}}) < \infty$, where $\gamma^i_{u^{i,*}} \equiv u^{i,*}$.
\end{itemize}
\end{assumption}

\begin{remark}
Note that Assumption~\ref{newas1}-(d),(e) hold if the cost function is bounded. Therefore, if the static team problem satisfies Assumption~\ref{as1}, then Assumption~\ref{newas1} (except (b)) holds as well. Hence, the results derived in this section \emph{almost} includes the results proved in Section~\ref{boundedcase} as a particular case. However, since the analysis in this section is somewhat involved, we presented the compact and bounded case as a separate section.
\end{remark}

The following result states that there exists a near optimal strategy whose range is convex and compact. In what follows, for any compact subset $G$ of $\sU$, we let
\begin{align}
{\bf \Gamma}_G \coloneqq \bigl\{ \underline{\gamma} \in {\bf \Gamma}: \underline{\gamma}(\sY) \subset G \bigr\}. \nonumber
\end{align}

\begin{lemma}\label{newprop2}
Suppose Assumption~\ref{newas1}-(a),(b),(c),(e) hold. Then, for any $\varepsilon>0$, there exists a compact subset $G$ of $\sU$ of the form $G = \prod_{i=1}^N G^i$, where each $G^i$ is convex and compact, such that
\begin{align}
\inf_{\underline{\gamma} \in {\bf \Gamma}_G} J(\underline{\gamma}) < J^* + \varepsilon. \nonumber
\end{align}
\end{lemma}

\begin{proof}
Fix any $\varepsilon>0$. Let $\underline{\gamma} \in {\bf \Gamma}$ with $J(\underline{\gamma}) < J^* + \varepsilon/2$. We construct the desired $G$ iteratively.

By Assumption~\ref{newas1}-(e) there exists $u^{1,*} \in \sU^1$ such that $J(\underline{\gamma}^{-1},\gamma^1_{u^{1,*}})<\infty$. Let $G^1 \subset \sU^1$ be a compact set containing $u^{1,*}$. We define
\begin{align}
\tgamma(y^1) = \begin{cases}
\gamma^1(y^1),   &\text{ if } \gamma^1(y^1) \in G^1  \\
u^{1,*} ,  &\text{ otherwise}.
\end{cases}\nonumber
\end{align}
Define also $\underline{\gamma}_1 \coloneqq (\tgamma^1,\gamma^2,\ldots,\gamma^N)$, $M_1 \coloneqq \bigl\{ y^1 \in \sY^1: \gamma^1(y^1) \in G^1 \bigr\}$, and $\tilde{u}^1 = \tgamma(y^1)$. Then, we have
\begin{align}
&|J(\underline{\gamma}) - J(\underline{\gamma}_1)| \nonumber \\
&= \biggl| E\bigl[ c(x,\by,\bu) 1_{\{y^1 \in M_1\}} \bigr] + E\bigl[ c(x,\by,\bu) 1_{\{y^1 \notin M_1\}} \bigr] \nonumber \\
&\phantom{xxxxxxxxxxx}- E\bigl[ c(x,\by,\bu^{-1},\tilde{u}^1) 1_{\{y^1 \in M_1\}} \bigr] \nonumber \\
&\phantom{xxxxxxxxxxxxxxxxxx}- E\bigl[ c(x,\by,\bu^{-1},\tilde{u}^1) 1_{\{y^1 \notin M_1\}} \bigr] \biggr| \nonumber \\
&\leq E\bigl[ c(x,\by,\bu) 1_{\{y^1 \notin M_1\}} \bigr] + E\bigl[ c(x,\by,\bu^{-1},\tilde{u}^1) 1_{\{y^1 \notin M_1\}} \bigr] \nonumber \\
&= \int_{\sX\times\sY\times\sU^{-1}\times (G^1)^c} c(x,\by,\bu) \text{ } \delta_{\underline{\gamma}}(d\bu) \text{ } \rP(dx,d\by) \nonumber \\
&\phantom{xxx}+ \int_{\sX\times\sY\times\sU^{-1}\times (G^1)^c} c(x,\by,\bu) \text{ } \delta_{(\underline{\gamma}^{-1},\gamma^1_{u^{1,*}})}(d\bu) \text{ } \rP(dx,d\by) \nonumber,
\end{align}
where $D^c$ denotes the complement of the set $D$, $\delta_z$ denotes the point mass at $z$, and $\sU^{-1} = \prod_{i=2}^N \sU^i$. Recall that $J(\underline{\gamma}^{-1},\gamma^1_{u^{1,*}}) < \infty$ by Assumption~\ref{newas1}-(e). Hence, the last expression can be made smaller than $\frac{\varepsilon}{2N}$ by properly choosing $G^1$ since $\sU^1$ is a Borel space \cite[Theorem 3.2]{Par67}. Since the closed convex hull of the set $G^1$ is compact by Theorem~\ref{clconv}, we can indeed take $G^1$ convex without loss of generality. By replacing $\underline{\gamma}$ with $\underline{\gamma}_1$ and applying the same method as above, we can obtain $\underline{\gamma}_2$, and a convex and compact $G^2 \subset \sU^2$ such that $|J(\underline{\gamma}_1)- J(\underline{\gamma}_2)| \leq \frac{\varepsilon}{2N}$ and $\gamma^2(\sY^2) \subset G^2$.

Continuing this way , we obtain $G = \prod_{i=1}^N G^i$ and $\underline{\gamma}_N \in {\bf \Gamma}_G$ such that $\bigl| J(\underline{\gamma}) - J(\underline{\gamma}_N) \bigr| < \varepsilon/2$, where $G^i$ is convex and compact for all $i=1,\ldots,N$. Hence, we have $J(\underline{\gamma}_N) < J^* + \varepsilon$, completing the proof.
\end{proof}


Recall that ${\bf \Gamma}_c$ denotes the set of continuous strategies. For any $G \subset \sU$, we define ${\bf \Gamma}_{c,G} \coloneqq {\bf \Gamma}_c \cap {\bf \Gamma}_G$; that is, ${\bf \Gamma}_{c,G}$ is the set of continuous strategies having range inside $G$.

\begin{proposition}\label{newprop1}
Suppose Assumption~\ref{newas1} holds. Then, for any $\varepsilon>0$, there exists a compact subset $G$ of $\sU$ of the form $G = \prod_{i=1}^N G^i$, where each $G^i$ is convex and compact, such that
\begin{align}
\inf_{\underline{\gamma} \in {\bf \Gamma}_{c,G}} J(\underline{\gamma}) < J^* + \varepsilon. \nonumber
\end{align}
\end{proposition}

\begin{proof}

Fix any $\varepsilon>0$. By Lemma~\ref{newprop2}, there exists a compact subset $G = \prod_{i=1}^N G^i$ of $\sU$, where $G^i$ is convex and compact, and $\underline{\gamma} \in {\bf \Gamma}_G$ such that
\begin{align}
J(\underline{\gamma}) < J^* + \frac{\varepsilon}{2}. \nonumber
\end{align}
Recall that $\mu^i$ denotes the distribution of $y^i$.

Let $\delta > 0$ which will be specified later. Analogous to the proof of Proposition~\ref{prop1}, we construct a continuous strategy $\underline{\tgamma}$ which is \emph{almost} equal to $\underline{\gamma}$. For each $i$, Lusin's theorem implies the existence of a closed set $F_{i} \subset \sY^i$ such that $\mu^i\bigl( \sY^i \setminus F_{i} \bigr) < \delta$ and the restriction, denoted by $\pi^i$, of $\gamma^i$ on $F_{i}$ is continuous.
By the Dugundji extension theorem there exists a continuous extension $\tgamma^i: \sY^i \rightarrow G^i$ of $\pi^i$. Therefore, $\underline{\tgamma} = (\tgamma^1,\ldots,\tgamma^N) \in {\bf \Gamma}_{c,G}$. Let $F \coloneqq \prod_{i=1}^N F_{i}$. Then, we have
\begin{align}
\rP\bigl( \sX \times (\sY\setminus F) \bigr) &\leq \sum_{i=1}^N \rP\bigl( \sX \times \sY^{-i} \times (\sY^i \setminus F_i) \bigr) \nonumber \\
&\leq \sum_{i=1}^N \delta = N \delta. \nonumber
\end{align}
Hence, we obtain
\begin{align}
\bigl| J(\underline{\gamma}) - &J(\underline{\tgamma}) \bigr|
= \biggl| \int_{\sX \times \sY} \bigl[ c(x,{\bf y},\underline{\gamma}) - c(x,{\bf y},\underline{\tgamma}) \bigr] \rP(dx,d{\bf y}) \biggr| \nonumber \\
&\leq \int_{\sX \times (\sY \setminus F)} \bigl[ c(x,{\bf y},\underline{\gamma}) + c(x,{\bf y},\underline{\tgamma}) \bigr] \text{  }\rP(dx,d{\bf y}) \nonumber \\
&\leq  2 \int_{\sX \times (\sY \setminus F)} w_G(x,{\bf y}) \text{  }\rP(dx,d{\bf y}). \nonumber
\end{align}
By Assumption~\ref{newas1}-(d) $w_G$ is $\rP$-integrable so that the measure $w_G(x,\by) \rP(dx,d\by)$ is absolutely continuous with respect to $\rP$. Since $\rP\bigl(\sX \times (\sY \setminus F) \bigr) \rightarrow 0$ as $\delta \rightarrow 0$, we obtain
\begin{align}
\int_{\sX \times (\sY \setminus F)} w_G(x,{\bf y}) \text{  }\rP(dx,d{\bf y}) \rightarrow 0 \text{  as   } \delta \rightarrow 0.\nonumber
\end{align}
Since $J(\underline{\gamma}) < J^* + \frac{\varepsilon}{2}$, there exists a sufficiently small $\delta>0$ such that $J(\underline{\tgamma}) < J^* + \varepsilon$. This completes the proof.
\end{proof}

Since for each $i$, $\sY^i$ is a locally compact separable metric space, there exists a nested sequence of compact sets $\{\sK_l^i\}$ such that $\sK_l^i \subset \intr \sK_{l+1}^i$ and $\sY^i = \bigcup_{l=1}^{\infty} \sK_l^i$ \cite[Lemma 2.76]{AlBo06}, where $\intr D$ denotes the interior of the set $D$.

Recall that $d_i$ denotes the metric on $\sY^i$. For each $l\geq1$, let $\sY_{l,n}^i \coloneqq \bigl\{ y_{i,1}, \ldots, y_{i,i_{l,n}} \bigr\} \subset \sK_l^i$ be an $1/n$-net in $\sK_l^i$. Recall that if $\sY_{l,n}^i$ is an $1/n$-net in $\sK_l^i$, then  for any $y \in \sK_l^i$ we have
\begin{align}
\min_{z \in \sY_{l,n}^i} d_i(y,z) < \frac{1}{n}. \nonumber
\end{align}
For each $l$ and $n$, let $q_{l,n}^i: \sK_l^i \rightarrow \sY_{l,n}^i$ be the nearest neighborhood quantizer; that is, for any $y \in \sK_l^i$
\begin{align}
q_{l,n}^i(y) = \argmin_{z \in \sY_{l,n}^i} d_{i}(y,z), \nonumber
\end{align}
where ties are broken so that $q_{l,n}^i$ is measurable.
If $\sK_l^i=[-M,M] \subset \sY^i = \R$ for some $M \in \R_+$, the finite set $\sY_{l,n}^i$ can be chosen such that $q_{l,n}^i$ becomes an uniform quantizer. We let $Q_{l,n}^i: \sY^i \rightarrow \sY_{l,n}^i$ denote the extension of $q_{l,n}^i$ to $\sY^i$ given by
\begin{align}
Q_{l,n}^i(y) \coloneqq \begin{cases}
q_{l,n}^i(y), &\text{ if } y \in \sK_l^i, \\
y_{i,0}, &\text{ otherwise},
\end{cases} \nonumber
\end{align}
where $y_{i,0} \in \sY^i$ is some auxiliary element. Define $\Gamma_{l,n}^i = \Gamma^i \circ Q_{l,n}^i \subset \Gamma^i$; that is, $\Gamma_{l,n}^i$ is defined to be the set of all strategies $\tgamma^i \in \Gamma^i$ of the form $\tgamma^i = \gamma^i \circ Q_{l,n}^i$, where $\gamma^i \in \Gamma^i$. Define also ${\bf \Gamma}_{l,n} \coloneqq \prod_{i=1}^N \Gamma_{l,n}^i \subset {\bf \Gamma}$.
The following theorem states that an optimal (or almost optimal) policy can be approximated with arbitrarily small approximation error for the induced costs by policies in ${\bf \Gamma}_{l,n}$ for sufficiently large $l$ and $n$.

\begin{theorem}\label{newthm1}
For any $\varepsilon>0$, there exist $(l,n(l))$ and $\underline{\gamma} \in {\bf \Gamma}_{l,n(l)}$ such that
\begin{align}
J(\underline{\gamma}) < J^* + \varepsilon. \nonumber
\end{align}
\end{theorem}

\begin{proof}
By Proposition~\ref{newprop1}, there exists $\underline{\gamma} \in {\bf \Gamma}_{c,G}$ such that $J(\underline{\gamma}) < J^* + \varepsilon/4$, where $G = \prod_{i=1}^N G^i$ and each $G^i$ is convex and compact. For each $l$ and $n$, we define $\gamma^i_{l,n} \coloneqq \gamma^i \circ Q_{l,n}^i$ and $\underline{\gamma}_{l,n} = (\gamma^1_{l,n},\ldots,\gamma^N_{l,n})$. Define also $u^{i,*} \coloneqq \gamma^i(y_{i,0}) \in G^i$.

Let ${\cal N}^*$ denote the collection of all subsets of ${\cal N}$ except the empty set. For any $s \in {\cal N}^*$, we define
\begin{align}
u^{s,*} &\coloneqq \bigl( u^{i,*} \bigr)_{i\in s}, \text{  } \underline{\gamma}_{u^{s,*}} \coloneqq \bigl( \gamma^i_{u^{i,*}} \bigr)_{i\in s}, \text{   } \underline{\gamma}^{-s} \coloneqq \bigl( \gamma^i \bigr)_{i \notin s}, \nonumber \\
\intertext{ and   }
\sK_l^s &\coloneqq \prod_{i\in s} \bigl( \sK_l^i \bigr)^c \times \prod_{i \notin s} \sK_l^i. \nonumber
\end{align}
Recall that $\gamma^i_{u^{i,*}}$ is the strategy which maps any $y^i \in \sY^i$ to $u^{i,*}$.
Let $\sK_l = \prod_{i=1}^N \sK_l^i$ and observe that
\[(\sX \times \sK_l)^c = \bigcup_{s \in {\cal N}^*} \sK_l^s.\]

Note that since the range of the strategy $(\underline{\gamma}^{-s},\underline{\gamma}_{u^{s,*}})$ is contained in $G$,
we have $J(\underline{\gamma}^{-s},\underline{\gamma}_{u^{s,*}}) \leq \int_{\sX \times \sY} w_G(x,\by) \rP(dx,d\by) < \infty$ for all $s \in {\cal N}^*$ by Assumption~\ref{newas1}-(d). Hence, there exists a sufficiently large $l$ such that
\begin{align}
\biggl| J(\underline{\gamma}) - \int_{\sX\times\sK_l} c(x,\by,\underline{\gamma}) \rP(dx,d\by) \biggl| &\leq \frac{\varepsilon}{4}, \nonumber \\
\intertext{and}
\int_{\sX\times\sK_l^s} c(x,\by,\underline{\gamma}^{-s},\underline{\gamma}_{u^{s,*}}) \rP(dx,d\by) &\leq \frac{\varepsilon}{2^{N+1}}, \text{ } \text{for all } s \in {\cal N}^*. \nonumber
\end{align}

Let
$q_{l,n}^{-s} = \bigl( q_{l,n}^i \bigr)_{i\notin s}$. Then, we have
\begin{align}
&\limsup_{n\rightarrow\infty} | J(\underline{\gamma}) - J(\underline{\gamma}_{l,n})| \leq \biggl| J(\underline{\gamma}) - \int_{\sX\times\sK_l} c(x,\by,\underline{\gamma}) \text{ }d\rP \biggr| \nonumber \\
&\phantom{xx}+ \limsup_{n\rightarrow\infty} \biggl| \int_{\sX\times\sK_l} c(x,\by,\underline{\gamma}) \text{ }d\rP - \int_{\sX\times\sK_l} c(x,\by,\underline{\gamma}_{l,n}) \text{ }d\rP \biggr| \nonumber \\
&\phantom{xx}+ \sum_{s \in {\cal N}^*} \limsup_{n\rightarrow\infty} \int_{\sX\times\sK_l^s} c(x,\by,\underline{\gamma}^{-s} \circ q_{l,n}^{-s} ,\underline{\gamma}_{u^{s,*}}) \text{ }d\rP \nonumber \\
&\leq \frac{\varepsilon}{4} + \limsup_{n\rightarrow\infty} \biggl| \int_{\sX\times\sK_l} c(x,\by,\underline{\gamma}) \text{}d\rP - \int_{\sX\times\sK_l} c(x,\by,\underline{\gamma}_{l,n}) \text{}d\rP \biggr| \nonumber \\
&\phantom{xx}+ \sum_{s \in {\cal N}^*} \limsup_{n\rightarrow\infty} \biggl| \int_{\sX\times\sK_l^s} c(x,\by,\underline{\gamma}^{-s} \circ q_{l,n}^{-s},\underline{\gamma}_{u^{s,*}}) \text{ }d\rP \nonumber \\
&\phantom{xxxxxxxxxxxxxxxxxxxx}- \int_{\sX\times\sK_l^s} c(x,\by,\underline{\gamma}^{-s},\underline{\gamma}_{u^{s,*}}) \text{ }d\rP \biggr|  \nonumber \\
&\phantom{xxxxxxxxxxxxxxxxx}+ \sum_{s \in {\cal N}^*} \int_{\sX\times\sK_l^s} c(x,\by,\underline{\gamma}^{-s},\underline{\gamma}_{u^{s,*}}) \text{ }d\rP \nonumber \\
&\leq \frac{\varepsilon}{4} + \sum_{s \in {\cal N}^*} \frac{\varepsilon}{2^{N+1}} \nonumber \\
&\phantom{xx}+ \limsup_{n\rightarrow\infty} \biggl| \int_{\sX\times\sK_l} c(x,\by,\underline{\gamma}) \text{ }d\rP - \int_{\sX\times\sK_l} c(x,\by,\underline{\gamma}_{l,n}) \text{ }d\rP \biggr| \nonumber \\
&\phantom{xx}+ \sum_{s \in {\cal N}^*} \limsup_{n\rightarrow\infty} \biggl| \int_{\sX\times\sK_l^s} c(x,\by,\underline{\gamma}^{-s} \circ q_{l,n}^{-s},\underline{\gamma}_{u^{s,*}}) \text{ }d\rP \nonumber \\
&\phantom{xxxxxxxxxxxxxxxxxxxx}- \int_{\sX\times\sK_l^s} c(x,\by,\underline{\gamma}^{-s},\underline{\gamma}_{u^{s,*}}) \text{ }d\rP \biggr| . \nonumber
\end{align}
Note that in the last expression, the integrands in the third and fourth terms are less than $w_G$. Since $w_G$ is $\rP$-integrable by Assumption~\ref{newas1}-(d), $\gamma^i \circ q_{l,n}^i \rightarrow \gamma^i$ on $\sK_l^i$ as $n\rightarrow\infty$ (recall that $\gamma^i$ is continuous), and $c$ is continuous in $\bu$, the third and fourth terms in the last expression converge to zero as $n\rightarrow\infty$ by dominated convergence theorem. Hence, there exists a sufficiently large $n(l)$ such that the last expression becomes less than $3\varepsilon/4$. Therefore, $J(\underline{\gamma}_{l,n(l)}) < J^* + \varepsilon$, completing the proof.
\end{proof}

The above result implies that to compute a near optimal policy for the team problem it is sufficient to choose a strategy based on the quantized observations $\bigl( Q_{l,n}^1(y^1), \ldots, Q_{l,n}^N(y^N) \bigr)$ for sufficiently large $l$ and $n$. Furthermore, this nearly optimal strategy can have a compact range of the form $G = \prod_{i=1}^N G^i$, where $G^i$ is convex and compact for each $i=1,\ldots,N$. 
However, to obtain a result analogous to the Theorem~\ref{thm2}, we need to impose a further assumption. To this end, we first introduce a finite observation model. For each $(l,n)$, let $\sZ_{l,n}^i \coloneqq \{y_{i,0},y_{i,1},\ldots,y_{i,i_{l,n}}\}$ (i.e., the output levels of $Q_{l,n}^i$) and define the stochastic kernels $W_{l,n}^i(\,\cdot\,|x)$ on $\sZ_{l,n}^i$ given $\sX$ as follows:
\begin{align}
W_{l,n}^i(\,\cdot\,|x) &\coloneqq \sum_{j=0}^{i_{l,n}} W(\S_{i,j}^{l,n}|x) \delta_{y_{i,j}}(\,\cdot\,), \nonumber
\end{align}
where $\S_{i,j}^{l,n} \coloneqq \bigl\{ y \in \sY^i: Q_{l,n}^i(y) = y_{i,j} \bigr\}$. Let $\Pi_{n,l}^i \coloneqq \bigl\{\pi^i: \sZ_{l,n}^i \rightarrow \sU^i, \text{$\pi^i$ measurable}\bigr\}$ and  ${\bf \Pi}_{l,n} \coloneqq \prod_{i=1}^N \Pi_{l,n}^{i}$.
Define $J_{l,n}: {\bf \Pi}_{l,n} \rightarrow [0,\infty)$ as
\begin{align}
J_{l,n}(\underline{\pi}) \coloneqq \int_{\sX \times \sZ_{l,n}} c(x,{\bf y},{\bf u}) \rP_{l,n}(dx,d{\bf y}), \nonumber
\end{align}
where $\underline{\pi} = (\pi^1,\ldots,\pi^N)$, ${\bf u} = \underline{\pi}({\bf y})$, $\sZ_{l,n} = \prod_{i=1}^N \sZ_{l,n}^i$, and $\rP_{l,n}(dx,d{\bf y}) = \rP(dx) \prod_{i=1}^N W_{l,n}^i(dy^{i}|x)$. Note that the probability measure $\rP_{l,n}$ can also be treated as a
measure on $\sX \times \sY$. In this case, it is not difficult to prove that $\rP_{l,n}$ converges to $\rP$ weakly as $l,n \rightarrow\infty$. For any compact subset $G$ of $\sU$, we also define ${\bf \Pi}_{l,n}^G \coloneqq \{\underline{\pi} \in {\bf \Pi}_{l,n}: \underline{\pi}(\sZ_{l,n}) \subset G\}$.

\begin{assumption}\label{nnewas1}
For any compact subset $G$ of $\sU$ of the form $G = \prod_{i=1}^N G^i$, we assume that the function $w_G$ is uniformly integrable with respect to the measures $\{\rP_{l,n}\}$; that is,
\begin{align}
\lim_{R\rightarrow\infty} \sup_{l,n} \int_{\{w_G > R\}} w_G(x,{\bf y}) \text{ } d\rP_{l,n} = 0. \nonumber
\end{align}
\end{assumption}

\begin{lemma}\label{convergence-unbounded}
Let $\{\underline{\pi}_{l,n}\}$ be a sequence of strategies such that $\underline{\pi}_{l,n} \in {\bf \Pi}_{l,n}^G$, where $G= \prod_{i=1}^N G^i$ and each $G^i$ is convex and compact. For each $l$ and $n$, define $\underline{\gamma}_{l,n} \coloneqq \underline{\pi}_{l,n} \circ Q_{l,n}$, where $Q_{l,n} \coloneqq (Q_{l,n}^1,\ldots,Q_{l,n}^N)$. Then, we have
\begin{align}
\lim_{l,n\rightarrow\infty} |J_{l,n}(\underline{\pi}_{l,n}) - J(\underline{\gamma}_{l,n})| =  0. \nonumber
\end{align}
\end{lemma}
\begin{proof}
Let us introduce the following finite measures on $\sX \times \sY$:
\begin{align}
\mu_G(S) &\coloneqq \int_{S} w_G(x,{\bf y}) \text{ } d\rP, \nonumber \\
\mu_G^{l,n}(S) &\coloneqq \int_{S} w_G(x,{\bf y}) \text{ } d\rP_{l,n}.  \nonumber
\end{align}
Since $\rP_{l,n}$ converges to $\rP$ weakly, by \cite[Theorem 3.5]{Ser82} and Assumption~\ref{nnewas1} we have $\mu_G^{l,n} \rightarrow \mu_G$ weakly as $l,n \rightarrow \infty$. Hence, the sequence $\{\mu_G^{l,n}\}$ is tight. Therefore, there exists a compact subset $K$ of $\sX \times \sY$ such that $\mu_G(K^c)< \varepsilon/2$ and $\mu_G^{l,n}(K^c) < \varepsilon/2$ for all $l,n$. Then, we have
\begin{align}
&|J_{l,n}(\underline{\pi}_{l,n}) - J(\underline{\gamma}_{l,n})| \nonumber \\
&\phantom{xx}= \bigg| \int_{\sX \times \sY} c(x,Q_{l,n}({\bf y}),\underline{\gamma}_{l,n}) \text{ } d\rP - \int_{\sX \times \sY} c(x,{\bf y},\underline{\gamma}_{l,n}) \text{ } d\rP \biggr|  \nonumber \\
&\phantom{xx}\leq \int_{K} \bigl| c(x,Q_{l,n}({\bf y}),\underline{\gamma}_{l,n}) - c(x,{\bf y},\underline{\gamma}_{l,n}) \bigr| \text{ } d\rP \nonumber \\
&\phantom{xxxxxxxxxxxx}+ \int_{K^c} w_G(x,{\bf y}) \text{ } d\rP + \int_{K^c} w_G(x,{\bf y}) \text{ } d\rP_{l,n} \nonumber
\end{align}
The first term in the last expression goes to zero as $l,n\rightarrow\infty$ by dominated convergence theorem and the fact that $c$ is bounded and continuous in ${\bf y}$. The second term is less than $\varepsilon$. Since $\varepsilon$ is arbitrary, this completes the proof.
\end{proof}

The following theorem is the main result of this section which is a consequence of Theorem~\ref{newthm1}. It states that to compute a near optimal strategy for the original team problem, it is sufficient to compute an optimal (or an almost optimal policy if an optimal one does not exist) policy for the team problem described above.

\begin{theorem}\label{newthm2}
Suppose Assumptions~\ref{newas1} and \ref{nnewas1} hold. Then, for any $\varepsilon>0$, there exists a pair $(l,n(l))$ and a compact subset $G = \prod_{i=1}^N G^i$ of $\sU$ such that an optimal (or almost optimal) policy $\underline{\pi}^{*}$ in the set ${\bf \Pi}_{l,n(l)}^G$ for the cost $J_{l,n(l)}$ is $\varepsilon$-optimal for the original team problem when $\underline{\pi}^{*} = (\pi^{1*},\ldots,\pi^{N*})$ is extended to $\sY$ via $\gamma^i = \pi^{i*} \circ Q_{l,n(l)}^i$.
\end{theorem}

\begin{proof}
Fix any $\varepsilon>0$. By Lemma~\ref{newprop2} and Theorem~\ref{newthm1}, there exists compact subset $G$ of $\sU$ of the form $G=\prod_{i=1}^N G^i$ such that
\begin{align}
\lim_{l,n\rightarrow\infty} \inf_{\underline{\gamma} \in {\bf \Gamma}_{l,n} \cap {\bf \Gamma}_G} J(\underline{\gamma}) - J^* &<\varepsilon \nonumber.
\end{align}
For each $l,n\geq1$, let $\underline{\gamma}_{l,n} \in {\bf \Gamma}_{l,n} \cap {\bf \Gamma}_G$ be such that
$J(\underline{\gamma}_{l,n})< \inf_{\underline{\gamma} \in {\bf \Gamma}_{l,n} \cap {\bf \Gamma}_G} J(\underline{\gamma}) + 1/(n+l)$. Define $\underline{\pi}_{l,n}$ as the restriction of $\underline{\gamma}_{l,n}$ to the set $\sZ_{l,n}$.
Then, we have
\begin{align}
J^* + \varepsilon &\geq \lim_{l,n\rightarrow\infty} J(\underline{\gamma}_{l,n}) \nonumber \\
&= \lim_{l,n\rightarrow\infty} J_{l,n}(\underline{\pi}_{l,n}) \text{ } \text{(by Lemma~\ref{convergence-unbounded})} \nonumber \\
&\geq \limsup_{l,n\rightarrow\infty} \inf_{\underline{\pi} \in {\bf \Pi}_{l,n}^G} J_{l,n}(\underline{\pi}). \nonumber
\end{align}
For the reverse inequality, for each $l,n\geq1$, let $\underline{\pi}_{l,n} \in {\bf \Pi}_{l,n}^G$ be such that $J_{l,n}(\underline{\pi}_{l,n}) < \inf_{\underline{\pi} \in {\bf \Pi}_{l,n}^G} J_{l,n}(\underline{\pi}) + 1/(n+l)$. Define $\underline{\gamma}_{l,n} \coloneqq \underline{\pi}_{l,n} \circ Q_{l,n}$.
Then, we have
\begin{align}
\liminf_{l,n\rightarrow\infty} \inf_{\underline{\pi} \in {\bf \Pi}_{l,n}^G} J_{l,n}(\underline{\pi}) &= \liminf_{l,n\rightarrow\infty} J_{l,n}(\underline{\pi}_{l,n}) \nonumber \\
&= \liminf_{l,n\rightarrow\infty} J(\underline{\gamma}_{l,n}) \text{ } \text{(by Lemma~\ref{convergence-unbounded})} \nonumber \\
&\geq J^*. \nonumber
\end{align}
This completes the proof.
\end{proof}

\section{Approximation of Dynamic Team Problems}\label{sec3}
The results for the static case apply also to the dynamic case, through a static reduction.
\subsection{Static Reduction of Sequential Dynamic Teams and Witsenhausen's Equivalent Model} \label{staticred}
First we review the equivalence between sequential dynamic teams and their static reduction (this is called {\it the equivalent model} \cite{wit88}\index{Witsenhausen's Equivalent Model}). Consider a dynamic team setting according to the intrinsic model where there are $N$ decision epochs, and Agent~$i$ observes $y^i \sim \eta_i(\,\cdot\,|\omega,{\bf u}^{i-1})$, and the decisions are generated as $u^i=\gamma^i(y^i)$. The resulting cost under a given team policy $\underline{\gamma}$ is
\[J(\underline{\gamma}) = E[c(\omega,{\bf y},{\bf u})].\]
This dynamic team can be converted to a static team provided that the following absolute continuity condition holds.

\begin{assumption}\label{as2}
For every $i=1,\ldots,N$, there exists a function $f_i: \Omega\times\sU_1^{i-1}\times\sY^i \rightarrow [0,\infty)$, where $\sU_1^{i-1} \coloneqq \prod_{j=1}^{i-1} \sU^i$, and a probability measure $Q_i$ on $\sY^i$ such that for all $S \in {\cal Y}^i$ we have
\begin{align}
\eta_i(S | \omega,{\bf u}^{i-1}) = \int_{S} f_i(\omega,{\bf u}^{i-1},y^i) Q_i(dy^i). \nonumber
\end{align}
\end{assumption}

Therefore, for a fixed choice of $\underline{\gamma}$, the joint distribution of $(w,{\bf y})$ is given by
\[\rP(d\omega,d{\bf y}) = \rP(d\omega) \prod_{i=1}^N f_i(\omega,{\bf u}^{i-1},y^i) Q_i(dy^i),\]
where $\bu^{i-1} = \bigl(\gamma^1(y^1),\ldots,\gamma^{i-1}(y^{i-1})\bigr)$. The cost function $J(\underline{\gamma})$ can then be written as
\begin{align}
J(\underline{\gamma}) &= \int_{\Omega\times \sY} c(\omega,{\bf y},{\bf u}) \rP(d\omega) \prod_{i=1}^N f_i(\omega,{\bf u}^{i-1},y^i) Q_i(dy^i) \nonumber \\
&= \int_{\Omega\times \sY} \tilde{c}(\omega,{\bf y},{\bf u}) \widetilde{\rP}(d\omega,d{\bf y}), \nonumber
\end{align}
where $\tilde{c}(\omega,{\bf y},{\bf u}) \coloneqq c(\omega,{\bf y},{\bf u})\prod_{i=1}^N f_i(\omega,{\bf u}^{i-1},y^i)$ and
$\widetilde{\rP}(d\omega,d{\bf y}) \coloneqq \rP(d\omega)\prod_{i=1}^N Q_i(dy^i)$. The observations now can be regarded as independent, and by incorporating the $f_i$ terms into $c$, we can obtain an equivalent {\it static team} problem. Hence, the essential step is to appropriately adjust the probability space and the cost function.

\begin{remark}
Note that in the static reduction method, some nice properties (such as continuity and boundedness) of the cost function $c$ of the original dynamic team problem can be lost, if the $f_i$ functions in Assumption~\ref{as2} are not well-behaved. However, the observation channels between $w$ and the $y^i$ are quite well-behaved for most of the practical models (i.e, additive Gaussian channel) which admits static reduction. Therefore, much of the nice properties of the cost function are preserved for most of the practical models.
\end{remark}

\subsection{Approximation Results for Nonclassical Dynamic Teams Admitting a Static Reduction}

The next theorem is the main result of this section. It states that for a class of dynamic team problems, finite models can approximate an optimal policy with arbitrary precision. In what follows, $\tilde{\rP}_{l,n}$ denotes the distribution of state and observations in the finite model approximation of the static reduction.

\begin{theorem}\label{thm4}
Suppose Assumptions~\ref{newas1}-(a),(b),(c),(e) and \ref{as2} hold. In addition, for each $i=1,\ldots,N$, $f_i(w,{\bf u}^{i-1},y^i)$ is continuous in ${\bf u}^{i-1}$ and $y^i$,
and
\begin{align}
\sup_{\bu \in G} c(w,\by,\bu) \prod_{i=1}^N f_i(w,{\bf u}^{i-1},y^i) \text{ is } \tilde{\rP}_{l,n}-\text{uniformly integrable} \nonumber 
\end{align}
for all compact $G \subset \sU$ of the form $G = \prod_{i=1}^N G^i$. Then, the static reduction of the dynamic team model satisfies Assumptions~\ref{newas1} and \ref{nnewas1}. Therefore, Theorems~\ref{newthm1} and \ref{newthm2} hold for the dynamic team problem. In particular, Theorems~\ref{newthm1} and \ref{newthm2} hold for the dynamic team problems satisfying Assumptions~\ref{newas1}, \ref{nnewas1} and \ref{as2}, if $f_i$ is bounded and continuous in ${\bf u}^{i-1}$ and $y^i$ for each $i=1,\ldots,N$.
\end{theorem}

\subsection{Approximation of Partially Nested Dynamic Teams}

An important dynamic information structure is the {\it partially nested} information structure. An IS is  partially nested if whenever $u^k$ affects $y^i$ for some $k<i$, Agent~i has access to $y^k$; that is, there exists a measurable function $f_{i,k}: \sY^i \rightarrow \sY^k$ such that $f_{i,k}(y^i) = y^k$ for all $\underline{\gamma} \in {\bf \Gamma}$ and all realizations of $\omega$.
For such team problems, one talks about {\it precedence relationships} among agents: Agent~$k$ is {\it precedent} to Agent~$i$ (or Agent~$k$  {\it communicates} to  Agent~$i$), if the former agent's actions affect the  information of the latter, in which case (to be partially nested) Agent~$i$ has to have the information based on which the action-generating policy of Agent~$k$ was constructed.

Dynamic teams with such an information structure always admit a static reduction through an informational equivalence. For such partially nested (or quasi-classical) information structures, a static reduction was studied by Ho and Chu in the context of LQG systems \cite{HoChu} and for a class of non-linear systems satisfying restrictive invertibility properties \cite{ho1973equivalence}.

For such dynamic teams, the cost function does not change as a result of the static reduction, unlike in the static reduction in Section~\ref{staticred}. Therefore, if the partially nested dynamic team satisfies Assumptions~\ref{newas1} and \ref{nnewas1}, then its static reduction also satisfies it. Hence, Theorems~\ref{newthm1} and \ref{newthm2} hold for such problems.

Before proceeding to the next section, we prove an auxiliary result which will be used in the next two sections.

\begin{lemma}\label{uniform}
Let $f$ and $g$ be non-negative real functions defined on metric spaces $\sE_1$ and $\sE_2$, respectively. Suppose
\begin{align}
\lim_{R\rightarrow\infty} \sup_{n\geq1} \int_{\{f > R\}} f \text{ } d\mu_n &= 0 \nonumber \\
\lim_{R\rightarrow\infty} \sup_{n\geq1} \int_{\{g > R\}} g \text{ } d\nu_n &= 0 \nonumber
\end{align}
for some sequence of probability measures $\{\mu_n\}$ and $\{\nu_n\}$. Then, we have
\begin{align}
\lim_{R\rightarrow\infty} \sup_n \int_{\{f g > R\}} f(e_1) g(e_2) \text{ } d\mu_n\otimes d\nu_n &= 0. \nonumber
\end{align}
\end{lemma}

\begin{proof}
Let $E_n[f] \coloneqq \int f d\mu_n$ and $\hat{E}_n[g] \coloneqq \int g d\nu_n$. It is easy to prove that $\sup_n E_n[f] \eqqcolon a < \infty$ and $\sup_n \hat{E}_n[g] \eqqcolon b < \infty$.
Note that $\{f g > R\} \subset \{f > \sqrt{R}\} \cup \{g > \sqrt{R}\}$. Hence,
\begin{align}
&\int_{\{f g > R\}} f(e_1) g(e_2) \text{ } d\mu_n\otimes d\nu_n \nonumber \\
&\phantom{xxxx}\leq \int_{\{f > \sqrt{R}\}} f g \text{ } d\mu_n\otimes d\nu_n + \int_{\{g > \sqrt{R}\}} f g \text{ } d\mu_n\otimes d\nu_n \nonumber \\
&\phantom{xxxx}= \hat{E}_n[g] \int_{\{f > \sqrt{R}\}} f \text{ } d\mu_n + E_n[f] \int_{\{g > \sqrt{R}\}} g \text{ } d\nu_n \nonumber \\
&\phantom{xxxx}\leq b \int_{\{f > \sqrt{R}\}} f \text{ } d\mu_n + a \int_{\{g > \sqrt{R}\}} g \text{ } d\nu_n. \nonumber
\end{align}
Since the last term converges to zero as $R\rightarrow\infty$ by assumption, this completes the proof.
\end{proof}


\section{Approximation of Witsenhausen's Counterexample and Asymptotic Optimality of Quantized Policies}\label{sec4}

\subsection{Witsenhausen's Counterexample and its Static Reduction}
In Witsenhausen's celebrated counterexample \cite{wit68} (see Fig.~\ref{fig1}), thare are two decision makers: Agent~$1$ observes a zero mean and unit variance Gaussian random variable $y^1$ and decides its strategy $u^1$. Agent~$2$ observes $y^2 \coloneqq u^1 + v$, where $v$ is a standard (zero mean and unit variance) Gaussian noise independent of $y^1$, and decides its strategy $u^2$.

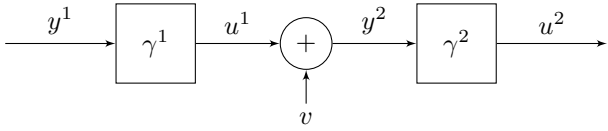
\begin{figure}[h]
\centering
\tikzstyle{int}=[draw, fill=white!20, minimum size=3em]
\tikzstyle{init} = [pin edge={to-,thin,black}]
\tikzstyle{sum} = [draw, circle]
\begin{tikzpicture}[node distance=2cm,auto,>=latex']
    \node [int] (a) {$\gamma^1$};
    \node (b) [left of=a,node distance=2cm, coordinate] {a};
    \node [sum] (d) [right of=a] {$+$};
    \node [int] (c) [right of=d] {$\gamma^2$};
    \node [coordinate] (end) [right of=c, node distance=2cm]{};
    \node (e) [below of=d, node distance=1cm] {$v$};
    \path[->] (e) edge node {} (d);
    \path[->] (b) edge node {$y^1$} (a);
    \path[->] (a) edge node {$u^1$} (d);
    \path[->] (d) edge node {$y^2$} (c);
    \path[->] (c) edge node {$u^2$} (end) ;
\end{tikzpicture}
\caption{Witsenhausen's counterexample.}
\label{fig1}
\end{figure}

The cost function of the team is given by
\begin{align}
c(y^1,u^1,u^2) = l (u^1 - y^1)^2 + (u^2 - u^1)^2, \nonumber
\end{align}
where $l \in \R_{+}$. Here, the state of the nature $\omega$ can be regarded as a degenerate random variable. We let $g(y) \coloneqq \frac{1}{\sqrt{2\pi}}\exp{\{-y^2/2\}}$. Then we have
\begin{align}
\rP(y^2 \in S|u^1) = \int_{S} g(y^2-u^1) m(dy^2), \nonumber
\end{align}
where $m$ denotes the Lebesgue measure. Let
\begin{align}
f(u^1,y^2) \coloneqq \exp{\bigl\{-\frac{(u^1)^2-2y^2u^1}{2}\bigr\}} \label{eq20}
\end{align}
so that $g(y^2-u^1) =f(u^1,y^2)\frac{1}{\sqrt{2\pi}}\exp{\{-(y^2)^2/2\}}$. The static reduction proceeds as follows: for any policy $\underline{\gamma}$, we have
\begin{align}
J(\underline{\gamma}) &= \int c(y^1,u^1,u^2) \rP(dy^2|u^1) \delta_{\gamma^1(y^1)}(du^1) \rP_{\sg}(dy^1) \nonumber \\
&= \int c(y^1,u^1,u^2) f(u^1,y^2) \rP_{\sg}(dy^2) \rP_{\sg}(dy^1), \nonumber
\end{align}
where $\rP_{\sg}$ denotes the standard Gaussian distribution. Hence, by defining $\tilde{c}(y^1,y^2,u^1,u^2) = c(y^1,u^1,u^2) f(u^1,y^2)$ and $\widetilde{\rP}(dy^1,dy^2) = \rP_{\sg}(dy^1) \rP_{\sg}(dy^2)$, we can write $J(\underline{\gamma})$ as
\begin{align}
J(\underline{\gamma}) = \int \tilde{c}(y^1,y^2,u^1,u^2) \widetilde{\rP}(dy^1,dy^2). \label{eq4}
\end{align}
hence, in the static reduction of Witsenhausen's counterexample, agents observe independent zero mean and unit variance Gaussian random variables.

\subsection{Approximation of Witsenhausen's Counterexample}

In this section we study the approximation problem for Witsenhausen's counterexample by using the static reduction formulation.
We show that the conditions in Theorem~\ref{thm4} hold for Witsenhausen's counterexample, and therefore, Theorems~\ref{newthm1} and \ref{newthm2} can be applied.

The cost function of the static reduction is given by
\begin{align}
\tilde{c}(y^1,y^2,u^1,u^2) = \bigl(l(u^1 - y^1)^2 + (u^2 - u^1)^2\bigr) f(u^1,y^2), \nonumber
\end{align}
where $f(u^1,y^2)$ is given in (\ref{eq20}). Note that the strategy spaces of the original problem and its static reduction are identical, and same strategies induce the same team costs. Recall that $\rP_{\sg}$ denotes the standard Gaussian distribution.

\begin{lemma}\label{newlemma1}
For any $(\gamma^1,\gamma^2) \in \Gamma^1 \times \Gamma^2$ with $J(\gamma^1,\gamma^2) < \infty$, we have $E \bigl[ \gamma^i(y^i)^2 \bigr] < \infty$.
\end{lemma}

\begin{proof}
To prove the lemma, we use the original problem setup instead of its static reduction. Fix any strategy $(\gamma^1,\gamma^2)$ with finite cost. Since $E\bigl[ (u^1 - y^1)^2 \big]<\infty$ and $u^1 =(u^1-y^1) + y^1$, we have $E\bigl[ (u^1)^2 \big]<\infty$ via $E\bigl[ (u^1)^2 \big]^{1/2} \leq E\bigl[ (u^1-y^1)^2 \big]^{1/2}+E\bigl[ (y^1)^2 \big]^{1/2}$. Then, we also have $E\bigl[ (u^2)^2 \big]<\infty$ since $E\bigl[ (u^1-u^2)^2 \big]<\infty$ and $E\bigl[ (u^2)^2 \big]^{1/2} \leq E\bigl[ (u^1)^2 \big]^{1/2} + E\bigl[ (u^1-u^2)^2 \big]^{1/2}$, completing the proof.
\end{proof}

For any $l \in \R_+$, we define $L \coloneqq [-l,l]$. Let $q_{l,n}$ denote the uniform quantizer on $L$ having $n$ output levels; that is,
\begin{align}
q_{l,n}&:L \rightarrow \{y_1,\ldots,y_{n}\} \subset L \nonumber \\
\intertext{and}
q^{-1}(y_j) &= \biggl[ y_j - \frac{\tau}{2}, y_j + \frac{\tau}{2} \biggr), \nonumber
\end{align}
where $\tau = \frac{2l}{n}$. Let us extend $q_{l,n}$ to $\R$ by mapping $\R \setminus L$ to $y_0=0$. For each $(l,n)$, let $\sZ_{l,n} \coloneqq \{y_0,y_{1},\ldots,y_{n}\}$ (i.e., output levels of the extended $q_{l,n}$) and define the probability measure $\rP_{l,n}$ on $\sZ_{l,n}$ as
\begin{align}
\rP_{l,n}(y_i) &= \rP_{\sg}(q_{l,n}^{-1}(y_i)). \nonumber
\end{align}
Moreover, let $\Pi^i_{l,n} \coloneqq \{\pi^i: \sZ_{l,n} \rightarrow \sU^i, \text{ $\pi^i$ measurable}\}$ and define
\begin{align}
J_{l,n}(\pi^1,\pi^2) \coloneqq \sum_{j,i=0}^{n} \tilde{c}(y_i,y_j,\pi^1(y_i),\pi^2(y_j)) \rP_{l,n}(y_i) \rP_{l,n}(y_j). \nonumber
\end{align}

With the help of Lemma~\ref{newlemma1}, we now prove the following result.

\begin{proposition}\label{newprop3}
Witsenhausen's counterexample satisfies conditions in Theorem~\ref{thm4}.
\end{proposition}

\begin{proof}
Assumption~\ref{newas1}-(a),(b),(c) and Assumption~\ref{as2} clearly hold. To prove Assumption~\ref{newas1}-(e), we introduce the following notation. For any strategy $(\gamma^1,\gamma^2)$, we let $E_{\gamma^1,\gamma^2}$ denote the corresponding expectation operation. Pick $(\gamma^1,\gamma^2)$ with $J(\gamma^1,\gamma^2) < \infty$. Since \begin{align}
E_{\gamma^1,\gamma^2} \biggl[ E_{\gamma^1,\gamma^2} \biggl[ \bigl( u^1 - u^2 \bigr)^2 \biggl| u^1 \biggr] \biggr] = E_{\gamma^1,\gamma^2} \biggl[ \bigl( u^1 - u^2 \bigr)^2 \biggr] \nonumber
\end{align}
by the law of total expectation, there exists $u^{1,*} \in \sU^1$ such that $E_{\gamma^1,\gamma^2} \bigl[ \bigl( u^1 - u^2 \bigr)^2 \bigl| u^1 = u^{1,*} \bigr] < \infty$. Let $u^{2,*} = 0$. Then we have
\begin{align}
&J(\gamma^1,\gamma_{u^{2,*}}^2) = E_{\gamma^1,\gamma_{u^{2,*}}^2}\biggl[ l \bigl(u^1 - y^1\bigr)^2 + \bigl(u^1\bigr)^2 \biggr] \nonumber \\
&\phantom{xxxxx}=E_{\gamma^1,\gamma^2}\biggl[ l \bigl(u^1 - y^1\bigr)^2 + \bigl(u^1\bigr)^2 \biggr] < \infty \text{  (by Lemma~\ref{newlemma1})} \nonumber \\
\intertext{and}
&J(\gamma_{u^{1,*}}^1,\gamma^2) = E_{\gamma_{u^{1,*}}^1,\gamma^2}\biggl[ l \bigl(u^{1} - y^1\bigr)^2 + \bigl(u^{1} - u^2\bigr)^2 \biggr] \nonumber \\
&\phantom{xx}=E_{\gamma^1,\gamma^2}\biggl[ l \bigl(u^{1,*} - y^1\bigr)^2 \biggr] + E_{\gamma^1,\gamma^2}\biggl[\bigl(u^{1} - u^2\bigr)^2 \biggl| u^1 = u^{1,*} \biggr] \nonumber \\
&\phantom{xx}<\infty. \nonumber
\end{align}
Therefore, Assumption~\ref{newas1}-(e) holds.

Note that for the $\rP_{l,n} \otimes \rP_{l,n}$-uniform integrability condition, it is sufficient to consider compact sets of the form $[-M,M]^2$ for some $M \in \R_+$, since any compact set in $\R^2$ is contained in $[-M,M]^2$ for sufficiently large $M \in \R_+$. Let $M \in \R_+$. We have
\begin{align}
w_1(y^1) &\coloneqq \sup_{(u^1,u^2) \in [-M,M]^2} l(u^1-y^1)^2 + (u^2-u^1)^2 \nonumber \\
&= l \bigl( M + |y^1|)^2 + 4 M^2 \nonumber
\end{align}
and
\begin{align}
\sup_{(u^1,u^2) \in [-M,M]^2} f(u^1,y^2) &\leq \sup_{(u^1,u^2) \in [-M,M]^2} \exp{\{y^2 u^1\}} \nonumber \\
&= \exp{\{M |y^2|\}} \eqqcolon w_2(y^2). \nonumber
\end{align}
For functions $w_1$ and $w_2$, we have
\begin{align}
&\lim_{R \rightarrow \infty} \sup_{n,l} \int_{\{w_1(y^1) > R\}} w_1(y^1) \text{ } d\rP_{l,n} \nonumber \\
&\phantom{xxxx}= \lim_{R \rightarrow \infty} \sup_{n,l} \int_{\{w_1(y^1) > R\}} w_1(q_{l,n}(y^1)) \text{ } d\rP_{\sg} \nonumber \\
&\phantom{xxxx}\leq \lim_{R \rightarrow \infty} \int_{\{w_1(y^1) > R\}} \biggl[ l\bigl(M+(|y^1|+1)^2\bigr) + 4 M^2 \biggr] \text{ } d\rP_{\sg}=0 \label{nneq5} \\
\intertext{and}
&\lim_{R \rightarrow \infty} \sup_{n,l} \int_{\{w_2(y^2) > R\}} w_2(y^2) \text{ } d\rP_{l,n} \nonumber \\
&\phantom{xxxx}= \lim_{R \rightarrow \infty} \sup_{n,l} \int_{\{w_2(y^2) > R\}} w_2(q_{l,n}(y^2)) \text{ } d\rP_{\sg} \nonumber \\
&\phantom{xxxx}\leq \lim_{R \rightarrow \infty} \int_{\{w_2(y^2) > R\}} \exp{M(|y^2|+1)} \text{ } d\rP_{\sg}=0, \label{nneq6}
\end{align}
where (\ref{nneq5}) and (\ref{nneq6}) follow from the facts that $q_{l,n}(\R \setminus L) =0$ and the integrability of $w_1$ and $w_2$ with respect to the $\rP_{\sg}$. By Lemma~\ref{uniform}, the product $w_1 w_2$ is $\rP_{l,n} \otimes \rP_{l,n}$-uniformly integrable. Therefore, $\sup_{(u^1,u^2) \in [-M,M]^2} \tilde{c}(y^1,y^2,u^1,u^2)$ is also $\rP_{l,n} \otimes \rP_{l,n}$-uniformly integrable. Since $M$ is arbitrary, this completes the proof.
\end{proof}

Proposition~\ref{newprop3} and Theorem~\ref{thm4} imply that Theorems~\ref{newthm1} and \ref{newthm2} is applicable to Witsenhausen's counterexample. Therefore, an optimal strategy for Witsenhausen's counterexample can be approximated by strategies obtained from finite models. The theorem below is the main result of this section. It states that to compute a near optimal strategy for Witsenhausen's counterexample, it is sufficient to compute an optimal strategy for the problem with finite observations obtained through uniform quantization of the observation spaces.

\begin{theorem}\label{thm3}
For any $\varepsilon>0$, there exists $(l,n(l))$ and $m \in \R_+$ such that an optimal policy $(\pi^{1*},\pi^{2*})$ in the set $\Pi_{l,n(l)}^{1,M} \times \Pi_{l,n(l)}^{2,M}$ for the cost $J_{l,n(l)}$ is $\varepsilon$-optimal for Witsenhausen's counterexample when $(\pi^{1*},\pi^{2*})$ is extended to $\sY^1 \times \sY^2$ via $\gamma^i = \pi^{i*} \circ q_{l,n(l)}$, $i=1,2$, where $M \coloneqq [-m,m]$ and $\Pi_{l,n(l)}^{i,M} \coloneqq \{\pi^i \in \Pi_{l,n(l)}^{i}: \pi^i(\sZ_{l,n}) \subset M\}$.
\end{theorem}

\section{The Gaussian Relay Channel Problem and Asymptotic Optimality of Quantized Policies}\label{gaussrelay}
\subsection{The Gaussian Relay Channel Problem and its Static Reduction}
An important dynamic team problem which has attracted interest is the Gaussian relay channel problem (see Fig.~\ref{fig2}) \cite{LipsaMartins,zaidi2013optimal}. Here, Agent~$1$ observes a noisy version of the state $x$ which has Gaussian distribution with zero mean and variance $\sigma_{x}^2$; that is, $y^1 \coloneqq x + v^0$ where $v^0$ is a zero mean and variance $\sigma_0^2$ Gaussian noise independent of $x$. Agent~$1$ decides its strategy $u^1$ based on $y^1$. For $i=2,\ldots,N$, Agent~$i$ receives $y^i \coloneqq u^{i-1} + v^{i-1}$ (a noisy version of the decision $u^{i-1}$ of Agent~$i-1$), where $v^{i-1}$ is a zero mean and variance $\sigma_{i-1}^2$ Gaussian noise independent of $\{x,v^{1},\ldots,v^{i-2},v^{i},\ldots,v^{N-1}\}$, and decides its strategy $u^{i}$.

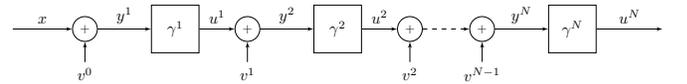
\begin{figure}[h]
\centering
\tikzstyle{int}=[draw, fill=white!20, minimum size=3em]
\tikzstyle{init} = [pin edge={to-,thin,black}]
\tikzstyle{sum} = [draw, circle, scale=0.8]
\scalebox{0.6}{
\begin{tikzpicture}[node distance=2cm,auto,>=latex']
    \node (x) [coordinate, node distance=2cm] {};
    \node [sum] (s0) [right of=x] {$+$};
    \node (v0) [below of=s0, node distance=1cm] {$v^0$};
    \node [int] (c1) [right of=s0] {$\gamma^1$};
    \node [sum] (s1) [right of=c1] {$+$};
    \node (v1) [below of=s1, node distance=1cm] {$v^1$};
    \node [int] (c2) [right of=s1] {$\gamma^2$};
    \node [sum] (s2) [right of=c2] {$+$};
    \node (v2) [below of=s2, node distance=1cm] {$v^2$};
    \node [sum] (s3) [right of=s2] {$+$};
    \node (v3) [below of=s3, node distance=1cm] {$v^{N-1}$};
    \node [int] (c3) [right of=s3] {$\gamma^{N}$};
    \node [coordinate] (end) [right of=c3, node distance=2cm] {};
    \path[->] (x) edge node {$x$} (s0);
    \path[->] (v0) edge node {} (s0);
    \path[->] (s0) edge node {$y^1$} (c1);
    \path[->] (c1) edge node {$u^1$} (s1);
    \path[->] (v1) edge node {} (s1);
    \path[->] (s1) edge node {$y^2$} (c2);
    \path[->] (c2) edge node {$u^2$} (s2);
    \path[->] (v2) edge node {} (s2);
    \path[->,dashed,->] (s2) edge node {} (s3);
    \path[->] (v3) edge node {} (s3);
    \path[->] (s3) edge node {$y^N$} (c3);
    \path[->] (c3) edge node {$u^N$} (end) ;
\end{tikzpicture}}
\caption{Gaussian relay channel.}
\label{fig2}
\end{figure}

The cost function of the team is given by
\begin{align}
c(x,\bu) \coloneqq \bigl( u^N - x \bigr)^2 + \sum_{i=1}^{N-1} l_i \bigl(u^i\bigr)^2, \nonumber
\end{align}
where $l_i \in \R_+$ for all $i=1,\ldots,N-1$. To ease the notation in the sequel, we simply take $\sigma_x = \sigma_0 = \sigma_1 = \ldots = \sigma_{N-1} =1$.
Recall that $g(y) \coloneqq \frac{1}{\sqrt{2\pi}} \exp{\{-y^2/2\}}$. Then we have
\begin{align}
\rP(y^1 \in S|x) &= \int_{S} g(y^1-x) m(dy^1)  \nonumber \\
\rP(y^i \in S|u^{i-1}) &= \int_{S} g(y^i-u^{i-1}) m(dy^{i}), \text{  for } i=2,\ldots,N. \nonumber
\end{align}
Recall also that $g(y-u) =f(u,y)\frac{1}{\sqrt{2\pi}}\exp{\{-(y)^2/2\}}$ , where $f(u,y)$ is defined in (\ref{eq20}).
Then, for any policy $\underline{\gamma}$, we have
\begin{align}
J(\underline{\gamma}) &= \int_{\sX \times \sY} c(x,\bu) \rP(dx,d\by) \nonumber \\
&= \int_{\sX\times\sY} c(x,\bu) \biggr[ f(x,y^1) \prod_{i=2}^N f(u^{i-1},y^i) \biggr] \text{ } \rP_{\sg}^{N+1}(dx,d\by), \nonumber
\end{align}
where $\rP_{\sg}^{N+1}$ denotes the product of $N+1$ zero mean and unit variance Gaussian distributions. Therefore, in the static reduction of Gaussian relay channel, we have the components $\tilde{c}(x,\by,\bu) \coloneqq c(x,\bu) \bigr[ f(x,y^1) \prod_{i=2}^N f(u^{i-1},y^i) \bigr]$ and $\widetilde{\rP}(dx,d\by) = \rP_{\sg}^{N+1}(dx,d\by)$. Analogous to Witsenhausen's counterexample, the agents observe independent zero mean and unit variance Gaussian random variables.

\subsection{Approximation of the Gaussian Relay Channel Problem}\label{gauss}

In this subsection, the approximation problem for the Gaussian relay channel is considered using the static reduction formulation.
Analogous to Section~\ref{sec4}, we prove that the conditions of Theorem~\ref{thm4} hold for Gaussian relay channel, and so Theorems~\ref{newthm1} and \ref{newthm2} can be applied.

The cost function of the static reduction is given by
\begin{align}
\tilde{c}(x,\by,\bu) \coloneqq c(x,\bu) \bigr[ f(x,y^1) \prod_{i=2}^N f(u^{i-1},y^i) \bigr] \nonumber
\end{align}
where $f(u,y)$ is given in (\ref{eq20}).

Recall the uniform quantizer $q_{l,n}$ on $L \coloneqq [l,-l]$ ($l \in \R_+$) having $n$ output levels from Section~\ref{sec4}.
We extend $q_{l,n}$ to $\R$ by mapping $\R \setminus L$ to $y_0=0$. Recall also the set $\sZ_{l,n} \coloneqq \{y_0,y_{1},\ldots,y_{n}\}$ and the probability measure $\rP_{l,n}$ on $\sZ_{l,n}$ given by
\begin{align}
\rP_{l,n}(y_i) &= \rP_{\sg}(q_{l,n}^{-1}(y_i)). \nonumber
\end{align}
Define $\Pi^i_{l,n} \coloneqq \{\pi^i: \sZ_{l,n} \rightarrow \sU^i, \text{ $\pi^i$ measurable}\}$ and
\begin{align}
J_{l,n}(\underline{\pi}) \coloneqq \int_{\sX} \sum_{\by \in \prod_{i=1}^N \sZ_{l,n}} \tilde{c}(x,\by,\underline{\pi}(\by)) \prod_{i=1}^N \rP_{l,n}(y_i) \rP_{\sg}(dx), \nonumber
\end{align}
where $\underline{\pi} \coloneqq (\pi^1,\ldots,\pi^N)$. Define $\tilde{\rP}_{l,n}(dx,d{\bf y}) \coloneqq \prod_{i=1}^N \rP_{l,n}(dy_i) \rP_{\sg}(dx)$.

\begin{proposition}\label{newprop4}
The Gaussian relay channel problem satisfies the conditions in Theorem~\ref{thm4}.
\end{proposition}

\begin{proof}
It is clear that Assumption~\ref{newas1}-(a),(b),(c) and Assumption~\ref{as2} hold. For any strategy $\underline{\gamma}$, let $E_{\underline{\gamma}}$ denote the corresponding expectation operation. To prove Assumption~\ref{newas1}-(e), pick $\underline{\gamma}$ with $J(\underline{\gamma})<\infty$. Analogous to Lemma~\ref{newlemma1}, one can prove that $E_{\underline{\gamma}} \bigl[ \bigl( u^N \bigr)^2 \bigr] < \infty$. For any $i=1,\ldots,N$, by the law of total expectation we can write
\begin{align}
J(\underline{\gamma}) &= E_{\underline{\gamma}} \biggl[ E_{\underline{\gamma}} \bigl[ c(x,\bu) \bigl| u^i \bigr] \biggr]  \nonumber \\
\intertext{and}
E_{\underline{\gamma}} \biggl[ \bigl( u^N \bigr)^2 \biggr] &= E_{\underline{\gamma}} \biggl[ E_{\underline{\gamma}} \bigl[ \bigl( u^N \bigr)^2 \bigl| u^i \bigr] \biggr]. \nonumber
\end{align}
Therefore, for each $i=1,\ldots,N$, there exists $u^{i,*} \in \sU^i$ such that $E_{\underline{\gamma}} \bigl[ c(x,\bu) \bigl| u^i = u^{i,*} \bigr] < \infty$ and $E_{\underline{\gamma}} \bigl[ \bigl( u^N \bigr)^2 \bigl| u^i = u^{i,*} \bigr]  < \infty$. Then we have
\begin{align}
&J(\underline{\gamma}^{-i},\gamma_{u^{i,*}}^i) = E_{\underline{\gamma}^{-i},\gamma_{u^{i,*}}^i}\biggl[ \bigl( u^N - x \bigr)^2 + \sum_{j=1}^{N-1} l_j \bigl(u^j\bigr)^2 \biggr] \nonumber \\
&\phantom{xxx}\leq E_{\underline{\gamma}^{-i},\gamma_{u^{i,*}}^i}\biggl[ 2 x^2 + 2 \bigl( u^N \bigr)^2 + \sum_{j=1}^{N-1} l_j \bigl(u^j\bigr)^2\biggr] \nonumber \\
&\phantom{xxx}=E_{\underline{\gamma}}\biggl[ \sum_{j=1}^{i-1} l_j \bigl(u^j\bigr)^2 + 2 x^2 \biggr] \nonumber \\
&\phantom{xxxxxxxx}+ E_{\underline{\gamma}}\biggl[ 2 \bigl( u^N \bigr)^2 + \sum_{j=i}^{N-1} l_j \bigl(u^j\bigr)^2 \biggr| u^i = u^{i,*} \biggr] <\infty. \nonumber
\end{align}
Therefore, Assumption~\ref{newas1}-(e) holds.

Analogous to the proof of Proposition~\ref{newprop3}, for $\tilde{\rP}_{l,n}$-uniform integrability, it is sufficient to consider compact sets of the form $[-M,M]^N$ for some $M \in \R_+$. Choose any $M \in \R_+$. Then we have
\begin{align}
w_1(x) &\coloneqq \sup_{\bu \in [-M,M]^N} \bigl( u^N - x \bigr)^2 + \sum_{i=1}^{N-1} l_i \bigl(u^i\bigr)^2 \nonumber \\
&= \bigl( M + x)^2 + \sum_{i=1}^{N-1} l_i M^2 \nonumber
\end{align}
and
\begin{align}
\sup_{\bu \in [-M,M]^N} \prod_{i=2}^N f(u^{i-1},y^i) &\leq \sup_{\bu \in [-M,M]^N} \prod_{i=2}^N  \exp{\{y^i u^{i-1}\}} \nonumber \\
&= \prod_{i=2}^N \exp{\{M |y^i|\}} \eqqcolon \prod_{i=2}^N w_{i,2}(y^i). \nonumber
\end{align}
It can be proved as in the proof of Proposition~\ref{newprop3} that $w_{i,2}$ is $\rP_{l,n}$-uniformly integrable for each $i=2,\ldots,N$. Let $\tilde{w}_1(x,y^1) \coloneqq w_1(x) f(x,y^1)$ and so
\begin{align}
\int_{\sX \times \sY^1} \tilde{w}_1(x,y^1) \exp{\{|x|\}} \text{ } d\rP_{\sg}^2 = \int_{\sX} w_1(x) \exp{\{|x|\}} \text{ } d\rP_{\sg} < \infty. \nonumber
\end{align}
Then, we have
\begin{align}
&\lim_{R\rightarrow\infty} \sup_{l,n} \int_{\{\tilde{w}_1 > R\}} \tilde{w}_1(x,y^1) \text{ } d\rP_{l,n} d\rP_{\sg} \nonumber \\
&\phantom{xxx}= \lim_{R\rightarrow\infty} \sup_{l,n} \int_{\{\tilde{w}_1 > R\}} \tilde{w}_1(x,q_{l,n}(y^1)) \text{ } d\rP_{\sg}^2 \nonumber \\
&\phantom{xxx}\leq \lim_{R\rightarrow\infty} \int_{\{\tilde{w}_1 > R\}} w_1(x) \exp{\bigl\{\frac{-x^2+2xy^1+2|x|}{2}\bigr\}} \text{ } d\rP_{\sg}^2 \nonumber \\
&\phantom{xxx}= \lim_{R\rightarrow\infty} \int_{\{\tilde{w}_1 > R\}} \tilde{w}_1(x,y^1) \exp{\{|x|\}} \text{ } d\rP_{\sg}^2 =0. \nonumber
\end{align}
Hence, by Lemma~\ref{uniform}, the product $\tilde{w}_1 \prod_{i=2}^N w_{i,2}$ is $\tilde{\rP}_{l,n}$-uniformly integrable. Therefore, $\sup_{{\bf u} \in [-M,M]^N} \tilde{c}(x,{\bf y},{\bf u})$ is also $\tilde{\rP}_{l,n}$-uniformly integrable. Since $M$ is arbitrary, this completes the proof.
\end{proof}

The preceding proposition and Theorem~\ref{thm4} imply, via Theorems~\ref{newthm1} and \ref{newthm2}, that an optimal strategy for Gaussian relay channel can be approximated by strategies obtained from finite models. The following theorem is the main result of this section.

\begin{theorem}\label{thm5}
For any $\varepsilon>0$, there exists $(l,n(l))$ and $m \in \R_+$ such that an optimal policy $\underline{\pi}^*$ in the set $\prod_{i=1}^N \Pi_{l,n(l)}^{i,M}$ for the cost $J_{l,n(l)}$ is $\varepsilon$-optimal for Gaussian relay channel when $\underline{\pi}^*$ is extended to $\sY$ via $\gamma^i = \pi^{i*} \circ q_{l,n(l)}$, $i=1,\ldots,N$, where $M \coloneqq [-m,m]$ and $\Pi_{l,n(l)}^{i,M} \coloneqq \{\pi^i \in \Pi_{l,n(l)}^{i}: \pi^i(\sZ_{l,n}) \subset M\}$.
\end{theorem}

\section{Discretization of the Action Spaces and Asymptotic Optimality of Finite Model Representations for Team Problems}\label{discact}

For computing near optimal strategies for static team problems using numerical algorithms, the action spaces $\sU^i$ must be finite. In this section, we show that the action spaces can be taken to be finite in finite observation models, if a sufficiently large number of points are used for accurate approximation. In this section, we consider the most general case studied in Section~\ref{unboundedcase}.

We note that the results that will be derived in this section can be applied to the dynamic teams which admit static reduction and satisfy conditions in Theorem~\ref{thm4}.

Recall the finite observation models constructed in Section~\ref{unboundedcase}. For each $(l,n)$, the finite model have the following components:
$\bigl\{ \sX, \sZ_{l,n}^i, \sU^i, W_{l,n}^i(\,\cdot\,|x), c, \rP, i \in {\cal N} \bigr\}$, where $\sX$ is the state space, $\sZ_{l,n}^i$ is the observation space for Agent~$i$, $\sU^i$ is the action space for Agent~$i$, $W_{l,n}^i(\,\cdot\,|x)$ observation channel from state to the observation of Agent~$i$, $c$ is the cost function, and $\rP$ is the distribution of the state. Furthermore, strategy spaces are defined as $\Pi_{n,l}^i \coloneqq \bigl\{\pi^i: \sZ_{l,n}^i \rightarrow \sU^i, \text{$\pi^i$ measurable}\bigr\}$ and  ${\bf \Pi}_{l,n} \coloneqq \prod_{i=1}^N \Pi_{l,n}^{i}$. Then the cost function $J_{l,n}: {\bf \Pi}_{l,n} \rightarrow [0,\infty)$ is given by
\begin{align}
J_{l,n}(\underline{\pi}) \coloneqq \int_{\sX \times \sZ_{l,n}} c(x,{\bf y},{\bf u}) \rP_{l,n}(dx,d{\bf y}), \nonumber
\end{align}
where $\underline{\pi} = (\pi^1,\ldots,\pi^N)$, ${\bf u} = \underline{\pi}({\bf y})$, $\sZ_{l,n} = \prod_{i=1}^N \sZ_{l,n}^i$, and $\rP_{l,n}(dx,d{\bf y}) = \rP(dx) \prod_{i=1}^N W_{l,n}^i(dy^{i}|x)$.




The theorem below is the main result of this section which states that one can approximate optimal strategy in ${\bf \Pi}_{l,n}$ by strategies taking values in a finite set.

\begin{theorem}\label{newthm3}
Suppose that original static team problem satisfies Assumptions~\ref{newas1} and \ref{nnewas1}. Then, for each $(l,n)$ and for any $\varepsilon>0$, there exist finite sets $\sU^i_{\varepsilon} \subset \sU^i$ for $i=1,\ldots,N$ such that
\begin{align}
\inf_{\underline{\pi} \in {\bf \Pi}_{l,n}^{\varepsilon}} J_{l,n}(\underline{\pi}) < J_{l,n}^* + \varepsilon, \nonumber
\end{align}
where ${\bf \Pi}_{l,n}^{\varepsilon} \coloneqq \bigl\{ \underline{\pi} \in {\bf \Pi}_{l,n}: \pi^i(\sZ_{l,n}^i) \subset \sU^i_{\varepsilon}, i \in {\cal N} \bigr\}$.
\end{theorem}

\begin{proof}
%

Fix any $(l,n)$ and $\varepsilon$. Let us choose $\underline{\pi}^{\varepsilon} \in {\bf \Pi}_{l,n}$ such that
\begin{align}
J_{l,n}(\underline{\pi}^{\varepsilon}) < \inf_{\underline{\pi} \in {\bf \Pi}_{l,n}} J_{l,n}(\underline{\pi}) + \frac{\varepsilon}{2}. \nonumber
\end{align}
Note that for any $i=1,\ldots,N$, the range of $\pi^{i,\varepsilon}$ is a finite subset of $\sU^i$ and so, is contained in some compact and convex subset $G^i$ of $\sU^i$. Define $G = \prod_{i=1}^N G^i$.

Let $\rho_i$ denote the metric on $\sU^i$. Since $G^i$ is compact, one can find a finite set $\sU^i_k \coloneqq \{u_{i,1},\ldots,u_{i,i_k}\} \subset G^i$ which is a $1/k$-net in $G^i$. Define $\Pi_k^i \coloneqq \bigl\{ \pi^i \in \Pi_{l,n}^{G^i} : \pi^i(\sZ_{l,n}^i) \subset \sU_k^i \bigr\}$ and ${\bf \Pi}_k = \prod_{i=1}^N \Pi_k^i$. For $\pi^{i,\varepsilon}$, we let
\begin{align}
\pi_k^{i,\varepsilon}(y) \coloneqq \argmin_{u \in \sU_k^i} \rho_i(\pi^{i,\varepsilon}(y),u). \nonumber
\end{align}
Hence
\begin{align}
\sup_{y \in \sZ_{l,n}^i} \rho_i(\pi^{i,\varepsilon}(y),\pi_k^{i,\varepsilon}(y)) < 1/k. \label{eq5}
\end{align}
We define $\underline{\pi}_k^{\varepsilon} = (\pi^{1,\varepsilon}_k,\ldots,\pi^{N,\varepsilon}_k)$.
Then we have
\begin{align}
&\inf_{\underline{\pi} \in {\bf \Pi}_{k}} J_{l,n}(\underline{\pi}) - \inf_{\underline{\pi} \in {\bf \Pi}_{l,n}} J_{l,n}(\underline{\pi}) < \inf_{\underline{\pi} \in {\bf \Pi}_{k}} J_{l,n}(\underline{\pi}) - J_{l,n}(\underline{\pi}^{\varepsilon}) + \frac{\varepsilon}{2} \nonumber \\
&\phantom{xxx}\leq J_{l,n}(\underline{\pi}_k^{\varepsilon}) - J_{l,n}(\underline{\pi}^{\varepsilon}) + \frac{\varepsilon}{2} \nonumber \\
&\phantom{xxx}\leq \int_{\sX\times\sZ_{l,n}} \bigl| c(x,\by,\underline{\pi}_k^{\varepsilon}) - c(x,\by,\underline{\pi}^{\varepsilon}) \bigr| \text{ } d\rP_{l,n} + \frac{\varepsilon}{2} \nonumber
\end{align}
The last integral converges to zero as $k\rightarrow\infty$ by the dominated convergence theorem since: (i) $c(x,\by,\underline{\pi}_k^{\varepsilon}) \rightarrow c(x,\by,\underline{\pi}^{\varepsilon})$ as $k\rightarrow\infty$ by (\ref{eq5}) and Assumption~\ref{newas1}-(a), (ii) $c(x,\by,\underline{\pi}_k^{\varepsilon}),c(x,\by,\underline{\pi}^{\varepsilon}) \leq w_G(x,\by)$ for all $(x,\by) \in \sX\times\sZ_{l,n}$, and (iii) $w_G$ is $\rP_{l,n}$-integrable. Therefore, there exists sufficiently large $k_0$ such that the last expression is less than $\varepsilon$.
By choosing $\sU_{\varepsilon}^i = \sU^i_{k_0}$, for $i=1,\ldots,N$, the proof is complete.
\end{proof}

Consider the finite {\it observation} models introduced in Section~\ref{sec4} that approximate the Witsenhausen's counterexample. For any $m \in \R_+$ and $k \in \R_+$, let $\sq_{m,k}: M \rightarrow \{u_1,\ldots,u_k\}$ denote the uniform quantizer with $k$ output levels (recall that $M \coloneqq [-m,m]$).
Note that $\sq_{m,k}$ is a quantizer applied to subsets of action spaces $\sU^i = \R$, $i=1,2$ (not to be confused with $q_{l,n}$ in Section~\ref{sec4}).
The preceding theorem implies that for each $(l,n)$ and $\varepsilon>0$, there exists a $m \in \R_+$ and $k \in \R_+$ such that
\begin{align}
\inf_{\underline{\pi} \in {\bf \Pi}_{l,n}^{m,k}} J_{l,n}(\underline{\pi}) < J_{l,n}^* + \varepsilon, \nonumber
\end{align}
where ${\bf \Pi}_{l,n}^{m,k} \coloneqq \bigl\{ \underline{\pi} \in {\bf \Pi}_{l,n}: \pi^i(\sZ_{l,n}^i) \subset \sU_{m,k}, i=1,2 \bigr\}$ and $\sU_{m,k} = \{u_1,\ldots,u_k\}$ is the set of output levels of $\sq_{m,k}$.


Therefore, to compute a near optimal strategy for Witsenhausen's counterexample, it is sufficient to compute an optimal strategy for the finite model that is obtained through uniform quantization of observation and action spaces (i.e., $\R$) on finite grids when the
number of grid points is sufficiently large. In particular, through constructing the uniform quantization so that both the {\it granular region} and the {\it granularity} of the quantizers are successively refined (that is the partitions generated by the quantizers are successively nested), we have the following proposition which lends itself to a numerical algorithm.

\begin{theorem}
There exists a sequence of finite models obtained through a successive refinement of the measurement and action set partitions generated by uniform quantizers whose optimal costs will converge to the cost of the Witsenhausen's counterexample.
\end{theorem}

\section{conclusion}\label{sec5}

Approximation of both static and dynamic team problems by finite models was considered. Under mild technical conditions, we showed that the finite model obtained by quantizing uniformly the observation and action spaces on finite grids provides a near optimal strategy if the number of grid points is sufficiently large. Using this result, an analogous approximation results were also established for the well-known counterexample of Witsenhausen and Gaussian relay channel. Our approximation approach to the Witsenhausen's counterexample thus provides, to our knowledge, the first rigorously established result that for any $\varepsilon > 0$, one can construct an $\varepsilon$ optimal strategy through an explicit solution of a conceptually simpler problem.

\section{Acknowledgements}

The authors are grateful to Professor Tamer Ba\c{s}ar for his technical comments and encouragement.


\end{document}